\RequirePackage{rotating}

\documentclass[11pt,letterpaper]{amsart}
\usepackage[dvipsnames,usenames]{color}
\usepackage{hyperref}
\usepackage{graphicx}
\usepackage{epsfig}
\usepackage[latin1]{inputenc}
\usepackage{amsmath}
\usepackage{amsfonts}
\usepackage{amssymb}
\usepackage{amsthm}
\usepackage{amscd}
\usepackage{verbatim}
\usepackage{subfigure}
\usepackage{pinlabel}
\usepackage{stmaryrd}
\usepackage{enumerate, enumitem}
\usepackage{makecell}

\usepackage[textsize=scriptsize]{todonotes}

\usepackage{tikz}
\usetikzlibrary{cd}
\usetikzlibrary{arrows}
\usetikzlibrary{decorations.pathreplacing}
\usepackage{verbatim}
	
\usepackage{rotating}

\newtheorem{theorem}{Theorem}[section]

\newtheorem{lemma}[theorem]{Lemma}
\newtheorem{proposition}[theorem]{Proposition}

\newtheorem{corollary}[theorem]{Corollary}

\theoremstyle{definition}
\newtheorem{definition}[theorem]{Definition}

\theoremstyle{remark}
\newtheorem{remark}[theorem]{Remark}
\newtheorem{example}[theorem]{Example}

\def\F{\mathbb{F}}

\def\Q{\mathbb{Q}}
\def\R{\mathbb{R}}
\def\Z{\mathbb{Z}}
\def\bbA{\mathbb{A}}
\def\bbB{\mathbb{B}}

\def\X{\mathbb{X}}

\def\cC{\mathcal{C}}

\def\II {\mathcal{I}}
\def\JJ {\mathcal{J}}
\def\CC{\mathcal{C}}

\newcommand{\abs}[1] {\left\lvert #1 \right\rvert}

\def\CF{\operatorname{CF}}

\def\CFh{\widehat{\operatorname{CF}}}
\def\HF{\operatorname{HF}}
\def\HFp{\operatorname{HF}^+}
\def\HFi{\operatorname{HF}^\infty}
\def\HFm{\operatorname{HF}^-}
\def\HFh{\widehat{\operatorname{HF}}}
\def\HFred{\operatorname{HF_{red}}}
\def\CFK{\operatorname{CFK}}
\def\CFKi{\operatorname{CFK}^{\infty}}

\def\CFKh{\widehat{\operatorname{CFK}}}

\def\HFKh{\widehat{\operatorname{HFK}}}

\def\Im{\operatorname{im}}

\DeclareMathOperator*{\connsum}{\raisebox{-0ex}{\scalebox{1.4}{$\#$}}}

\def\d{\partial}
\def\varep{\varepsilon}
\def\co{\colon}
\def\spinc{\textrm{Spin}^c}
\def\horz{\textup{horz}}

\def\T{\tilde{T}}
\def\CZhat{\widehat{\cC}_\Z}
\def\CZ{{\cC}_\Z}
\def\Df{\mathfrak{D}}

\def\conn {\mathbin{\#}}

\def\spincs {\mathfrak{s}}

\DeclareMathOperator{\im}{im}  
\DeclareMathOperator{\coker}{coker} \DeclareMathOperator{\lk}{lk} 
 \DeclareMathOperator{\Spin}{Spin}

 \DeclareMathOperator{\Char}{Char}

\DeclareMathOperator{\Cone}{Cone}
\DeclareMathOperator{\St}{St}

\def\x {\mathbf{x}}

\definecolor{darkgreen}{rgb}{0,0.5,0}
\definecolor{purple}{rgb}{0.5,0,0.5}

\renewcommand{\MR}[1]{}

\title{}

\subjclass[2013]{}
\author[Jennifer Hom]{Jennifer Hom}
\thanks{The first author was partially supported by NSF grant DMS-1552285 and a Sloan Research Fellowship.}
\address {School of Mathematics, Georgia Institute of Technology, Atlanta, GA 30332}
\email{hom@math.gatech.edu}

\author[Adam Simon Levine]{Adam Simon Levine}
\thanks{The second author was partially supported by NSF grants DMS-1405378 and DMS-1806437.}
\address{Department of Mathematics, Duke University, Durham, NC 27708}
\email{alevine@math.duke.edu}

\author[Tye Lidman]{Tye Lidman}
\thanks{The third author was partially supported by NSF grant DMS-1709702.}
\address{Department of Mathematics, North Carolina State University, Raleigh, NC 27607}
\email{tlid@math.ncsu.edu}

\numberwithin{equation}{section}

\title{Knot concordance in homology cobordisms}

\begin{document}
\maketitle

\begin{abstract}
Let $\CZhat$ denote the group of knots in homology spheres that bound homology balls, modulo smooth concordance in homology cobordisms. Answering a question of Matsumoto, the second author previously showed that the natural map from the smooth knot concordance group $\CC$ to $\CZhat$ is not surjective. Using tools from Heegaard Floer homology, we show that the cokernel of this map, which can be understood as the non-locally-flat piecewise-linear concordance group, is infinitely generated and contains elements of infinite order. In the appendix, we provide a careful proof that any piecewise-linear surface in a smooth $4$-manifold can be isotoped to be smooth away from cone points.
\end{abstract}


\section{Introduction}

Two oriented knots $K_0, K_1 \subset S^3$ are called \emph{smoothly concordant} if there is a smoothly, properly
embedded, oriented annulus $A$ in $S^3 \times [0,1]$ such that $(S^3 \times [0,1], A)$ is an oriented cobordism from $(S^3 \times \{0\}, K_0 \times \{0\})$ to $(S^3 \times \{1\}, K_1 \times \{1\})$. The  \emph{knot concordance group} $\cC$ consists of knots in $S^3$ modulo smooth concordance, under the operation induced by connected sum. The goal of this paper is to study various generalizations of the knot concordance group, where we consider knots in other $3$-manifolds and surfaces in $4$-dimensional cobordisms between them.

A \emph{homology cobordism} between two closed, oriented $3$-manifolds $Y_0, Y_1$ is a smooth, compact, oriented $4$-manifold $W$ with an identification $\partial W \cong -Y_0 \sqcup Y_1$, such that the inclusions $Y_i \rightarrow W$ induce isomorphisms on integral homology. We say that $Y_0$ and $Y_1$ are \emph{homology cobordant} if there is a homology cobordism between them. A \emph{homology $3$-sphere} is a closed, oriented $3$-manifold $Y$ such that $H_*(Y;\Z) \cong H_*(S^3;\Z)$. Here we will focus on those homology $3$-spheres which are homology cobordant to $S^3$, or equivalently those which bound homology $4$-balls; we call such manifolds \emph{homology null-cobordant}.

Formally, a \emph{knot} is a pair $(Y,K)$, where $Y$ is a closed, oriented $3$-manifold and $K$ is an oriented submanifold of $Y$ diffeomorphic to $S^1$. (For conciseness, we will often refer to $K$ as a knot when $Y$ is clear from context.) Two knots $(Y_0,K_0)$ and $(Y_1,K_1)$ are \emph{homology concordant}, denoted $(Y_0, K_0) \sim (Y_1, K_1)$, if there is a homology cobordism $W$ from $Y_0$ to $Y_1$ and a smoothly embedded annulus $A \subset W$ whose oriented boundary is identified with $-K_0 \sqcup K_1$ (under the chosen identification of $\partial W$ with $-Y_0 \sqcup Y_1$). Such a pair $(W,A)$ is called a \emph{homology concordance}. Note that if there is a diffeomorphism of pairs $(Y_0,K_0) \cong (Y_1,K_1)$, then $(Y_0,K_0)$ is homology concordant to $(Y_1,K_1)$.

Let $\CZ$ denote the group of knots in $S^3$, modulo homology concordance.\footnote{Our notation for $\CZ$ and $\CZhat$ follows Davis and Ray \cite{DavisRaySatellite}, who studied more general groups $\cC_R$ and $\widehat{\cC}_R$ of knots modulo concordance in $R$-homology cobordisms, where $R$ can be any localization of $\Z$. Here, we focus only on integer homology, although many of the results extend to rational homology as well.} A knot represents the trivial element in $\CZ$ if and only if it bounds a smoothly embedded disk in some homology $4$-ball. Note that $\CZ$ is naturally a quotient of $\cC$. It is unknown whether the quotient map $\cC \to \CZ$ is injective (i.e., whether a knot that bounds a disk in a homology $4$-ball must also bound a disk in the standard $4$-ball); this question is challenging because most familiar concordance invariants of knots in $S^3$ descend to $\CZ$.

Let $\CZhat$ denote the group of knots $(Y,K)$, where $Y$ is a homology null-cobordant homology $3$-sphere, modulo homology concordance. The group operation is induced by connected sum. To be precise, given knots $(Y_1,K_1)$ and $(Y_2,K_2)$, the connected sum operation (performed at points on $K_1$ and $K_2$) gives rise to a knot in the $3$-manifold $Y_1 \conn Y_2$, well-defined up to diffeomorphism of pairs. By slight abuse of notation, we will denote any representative of the resulting diffeomorphism class by either $(Y_1, K_1) \conn (Y_2, K_2)$ or $(Y_1 \conn Y_2, K_1 \conn K_2)$. The connected sum operation descends to give an abelian group structure on $\CZhat$. The identity element in $\CZhat$ is represented by knots that bound smooth disks in homology balls, or equivalently, knots that are homology concordant to the unknot in $S^3$.
The inverse of $(Y,K)$ is given by $-(Y, K) = (-Y,-K)$; that is, we reverse the orientation of the ambient manifold and also reverse the string orientation of the knot. There is a natural inclusion $\varphi \co \CZ \hookrightarrow \CZhat$. A knot $(Y,K)$ is in the image of $\varphi$ if and only if $(Y,K)$ is homology concordant to some knot in $S^3$.

Answering a question posed in the 1970s by Matsumoto \cite[Problem 1.31]{Kirby}, the second author showed in \cite{Levinenonsurj} that $\varphi$ is not surjective. Our main theorem builds on this result, as follows:

\begin{theorem}\label{thm:main}
The subgroup $\CZ \subset \CZhat$ is of infinite index. More specifically,
\begin{enumerate}
	\item \label{it:infgen} $\CZhat/\CZ$ is infinitely generated. Moreover, $\CZhat/\CZ$ cannot be generated by knots in any finite collection of $3$-manifolds.
	\item \label{it:Zsubgroup} $\CZhat/\CZ$ contains a subgroup isomorphic to $\Z$.
\end{enumerate}
\end{theorem}

Theorem \ref{thm:main} can be interpreted in terms of non-locally-flat piecewise-linear concordance.
As background, every knot $K \subset S^3$ bounds a piecewise-linear (PL) disk in $B^4$, obtained by
taking the cone over $K$. Such a disk is smooth (and hence locally flat) except at the cone point. Resolving a conjecture of Zeeman \cite{ZeemanDunce}, Akbulut \cite{AkbulutZeeman} proved that the same need not hold for an
arbitrary contractible $4$-manifold: he exhibited a contractible $4$-manifold $X$ and a knot $K
\subset \partial X$ such that $K$ does not bound a PL disk (even with singularities) in $X$. However, Akbulut observed that there is a self-diffeomorphism $h \co \partial X \to \partial X$ such that $h(K)$ bounds a smoothly embedded disk in $X$. We may thus view $K$ as bounding a smooth disk in a ``different'' contractible $4$-manifold, namely $X$ with an alternate parametrization of its boundary.

Note that a knot $K \subset Y$ bounds a PL disk in some homology $4$-ball $X$ if and only if $(Y,K)$ is in $\im(\varphi)$; this can be seen by adding or deleting neighborhoods of cone point singularities. The main result of \cite{Levinenonsurj} gives a pair $(Y,K)$ such that $Y$ bounds a contractible manifold $X$ but $K$ does not bound a PL disk in $X$ or in any other homology ball $X'$ with $\partial X' \cong Y$. By the same token, two pairs $(Y_0,K_0), (Y_1,K_1) \in \CZhat$ differ by an element of $\CZ$ if and only if $K_0$ and $K_1$ cobound a PL annulus in some homology cobordism from $Y_0$ to $Y_1$. Thus, the quotient $\CZhat/\CZ$ can be interpreted as the group of knots in homology null-cobordant homology spheres modulo PL concordance in homology cobordisms. (PL concordances in $Y \times I$ have also been studied under the name {\em almost concordance} by Celoria \cite{Celoria}.) Throughout this discussion, we have implicitly used the fact that a PL embedded surface in a smooth $4$-manifold can be taken to be smooth except at finitely many cone point singularities; see Appendix \ref{app:PL-smooth} below for a careful proof.

%
%
%

\begin{remark}
The arguments also apply to the group of knots in integer homology spheres that bound rational homology balls, modulo concordances in rational homology cobordisms.
\end{remark}

One of the main difficulties in understanding $\CZhat/\CZ$ is a paucity of invariants. Indeed, in order for  a concordance invariant to descend to an invariant on $\CZhat/\CZ$, we need the invariant to vanish on all knots in $S^3$, which is typically not a desired property of a knot invariant. Rather, our strategy for proving Theorem \ref{thm:main} is to study the relations among different numerical concordance invariants derived from Heegaard Floer homology. For a knot in $S^3$, certain of these invariants satisfy relations that need not hold for a knot in an arbitrary $3$-manifold, and the failure of any of these relations for a knot $K \subset Y$ gives an obstruction to $K$ being homology concordant to a knot in $S^3$. (These relations typically hold because the Heegaard Floer homology of $S^3$ is particularly simple.)

As an example of this approach, associated to any homology sphere $Y$, there is an even integer $d(Y)$ (defined by Ozsv\'ath and Szab\'o \cite{OSabsgr}) which is invariant under homology cobordism. If two knots are homology concordant, then their $r$-surgeries are homology cobordant for any $r \in \Q$; in particular, for each $n \in \Z$, $d(Y_{1/n}(K))$ is an invariant of the class of $(Y,K)$ in $\CZhat$. Ni and Wu \cite[Proposition 1.6]{NiWu} proved that for any knot $K \subset S^3$ and any positive integers $m$ and $n$,
\begin{equation}\label{eq:dsurg}
d(S^3_{1/m}(K)) = d(S^3_{1/n}(K)).
\end{equation}
Therefore, the same must be true for any knot $K \subset Y$ that is homology concordant to a knot in $S^3$. Critically, the proof of \eqref{eq:dsurg} relies on the fact that the ambient manifold is $S^3$ (or, more precisely, on the fact that the reduced Heegaard Floer homology of $S^3$ is trivial), and its failure to hold for knots in other $3$-manifolds gives a new, elementary proof of the non-surjectivity of $\varphi$:

\begin{example} \label{ex:T23}
Let $Y = \Sigma(2,3,13)$, thought of as $-1/2$-surgery on the right-handed trefoil $T_{2,3}$, and let $K \subset Y$ denote the core of the surgery solid torus. Akbulut and Kirby \cite{AkbulutKirby} showed that $Y$ bounds a contractible $4$-manifold.  For any $n \in \Z$, note that $Y_{1/n}(K) = S^3_{1/(n-2)}(T_{2,3})$; this is true because $T_{2,3}$ and $K$ have the same exterior, and the meridian of $K$ is given by $\mu - 2\lambda$, where $\lambda$ and $\mu$ are the longitude and meridian of $T_{2,3}$, respectively. In particular, we have:
\begin{align*}
d(Y_{1/2}(K)) &= d(S^3) = 0 \\
d(Y_{1/3}(K)) &= d(S^3_1(T_{2,3})) = d(\Sigma(2,3,5)) = -2.
\end{align*}
Since these are not equal, $K$ cannot be homology concordant to any knot in $S^3$. (The same argument works with $Y' = \Sigma(2,3,25) = S^3_{-1/4}(T_{2,3})$, which bounds a contractible $4$-manifold by work of Fickle \cite{Fickle}.)
\end{example}

Building on \eqref{eq:dsurg}, given a knot $K$ in a homology sphere $Y$, define
\[
\theta(Y,K) = \max_{m,n>0} \left\lvert d(Y_{1/m}(K)) - d(Y_{1/n}(K)) \right\rvert,
\]
which is finite by Proposition \ref{prop:thetabounds} below, and is evidently invariant under homology concordance. It is clear from \eqref{eq:dsurg} that for any knot $K \subset S^3$, we have $\theta(S^3, K)=0$. Moreover, we show that $\theta(Y,K)$ is bounded from above in terms of $\HFred(Y)$, the reduced Heegaard Floer homology of $Y$ \cite{OS3manifolds1}. Recall that $\HFred$ can be defined as the quotient of $\HFp$ by large powers of $U$, and hence all elements are in the kernel of some power of $U$.  Further, this module is a finite-dimensional $\F$-vector space, where $\F = \Z/2$.  Let
\[
N_Y = \min \{n \geq 0 \mid U^n \cdot \HFred(Y) = 0 \}.
\]
In Sections \ref{sec:dinvts} and \ref{sec:knot-family}, we will prove the following two results:

\begin{proposition}\label{prop:thetabounds}
Let $K$ be a knot in a homology sphere $Y$. Then $\theta(Y, K) \leq 2N_Y$.
\end{proposition}

\begin{proposition}\label{prop:theta-unbounded}
There exists a family of pairs $(Y_j,K_j)$ such that each $Y_j$ bounds a smooth contractible $4$-manifold and $\theta(Y_j,K_j)$ is unbounded as $j \to \infty$.
\end{proposition}

\begin{proof}[Proof of Theorem \ref{thm:main} \eqref{it:infgen}]
If $\CZhat/\CZ$ were generated by the classes of knots in a finite set of homology spheres $Z_1, \dots, Z_m$, then every knot in any homology null-cobordant homology sphere would be homology concordant to some knot of the form
\[
J \subset a_1 Z_1 \conn \dots \conn a_m Z_m,
\]
where $a_1, \dots, a_m \in \Z$.  We have that $N_Z = N_{-Z}$ by \cite[Proposition 2.5]{OS3manifolds2} and $N_{Z \# Z'} = \max\{N_Z, N_{Z'}\}$ by \cite[Theorem 1.5]{OS3manifolds2}.  This implies that
\begin{equation}\label{eq:kunneth}
N_{a_1 Z_1 \conn \dots \conn a_m Z_m} = \max\{N_{Z_i} \mid a_i \ne 0\}.
\end{equation}
Thus, Proposition \ref{prop:thetabounds} would give a universal bound on $\theta(Y,K)$ for all $(Y,K)\in \CZhat$, contradicting Proposition \ref{prop:theta-unbounded}.
\end{proof}

\begin{remark}
Since the $d$-invariants of surgery are not concordance homomorphisms, we are unable at present to show that the elements in our infinite generating set are of infinite order.
\end{remark}

The proof of the second item of Theorem \ref{thm:main} relies on two other invariants, $\tau$ and
$\varep$, coming from the Heegaard Floer homology package. Ozsv\'ath and Szab\'o \cite[Section 5]{OS4ball}
defined an invariant $\tau(Y,K) \in \Z$ associated to any knot $K$ in a homology sphere $Y$, and
they showed that it induces a group homomorphism $\tau \co \CZ \to \Z$; recent work of Raoux
\cite{Raouxtau} shows that $\tau$ is actually a homomorphism on all of $\CZhat$. For knots $K
\subset S^3$, the invariant $\varep(K) \in \{-1,0,1\}$, defined by the first author
\cite{Homconcordance}, is likewise a concordance invariant and satisfies a ``sign-additivity''
property under connected sums. It is also worthwhile to mention
Ozsv\'ath--Stipsicz--Szab\'o's $\Upsilon$ invariant, which associates to each $K \subset S^3$ a
continuous, piecewise linear function $\Upsilon_K \co [0,2] \to \R$ and induces a homomorphism
$\CC \to C^0([0,2],\R)$ \cite{OSS}. All three of these invariants can be understood in terms of filtrations on the knot Floer complex of $K$. In Section \ref{sec:concinv}, we prove:

\begin{theorem} \label{thm:concinv}
The invariants $\varep$ and $\Upsilon$ both extend to knots in arbitrary homology $3$-spheres, are
invariant under homology concordance, and satisfy the same additivity properties as for knots in
$S^3$.
\end{theorem}

\noindent See Propositions \ref{prop:epsilonprops} and \ref{prop:upsilonprops} for the precise statements regarding additivity.

For the purposes of studying $\CZhat/\CZ$, the key property of $\tau$ and $\varep$ is the following:

\begin{proposition}[{\cite[Proposition 3.6(2)]{Homcables}}]
Suppose $K$ is a knot in $S^3$. If $\varep(S^3,K)=0$, then $\tau(S^3,K)=0$.
\end{proposition}

\begin{corollary}\label{cor:obstruction}
If $\varep(Y,K)=0$ and $\tau(Y,K)\neq 0$, then $K \subset Y$ is not homology concordant to any knot in $S^3$.
\end{corollary}

In order to apply this obstruction, we need examples of knots with $\varep=0$ and $\tau \neq 0$. We prove the following in Section \ref{sec:mappingcone}:

\begin{proposition}\label{prop:epsilonzerotaunonzero}
Let $Z$ denote $+1$-surgery on the $(2,3)$-cable of the left-handed trefoil, let $Y = Z \conn {-Z}$,
and let $K \subset Y$ denote the connected sum of the core of the surgery in $Z$ with the unknot in
$-Z$. Then $(Y,K)$ represents an element of $\CZhat$, and:
\begin{enumerate}
\item $\tau(Y,K) = -1$;
\item $\varep(Y,K)=0$; and
\item $\Upsilon_{Y,K}$ is not identically $0$.
\end{enumerate}
\end{proposition}

\begin{proof}[Proof of Theorem \ref{thm:main} \eqref{it:Zsubgroup}]
By Proposition \ref{prop:tauepsilonhomcob}, $\tau$ and $\varep$ are invariants of homology concordance. Consider the pair $(Y,K)$ from Proposition \ref{prop:epsilonzerotaunonzero}. Since $\varep(Y,K)=0$ but $\tau(Y,K)=-1$, $(Y,K)$ is not homology concordant to any knot in $S^3$. By Propositions \ref{prop:tauprops} and \ref{prop:epsilonprops}, it follows that $\tau(\#_n(Y,K))=-n$ and $\varep(\#_n(Y,K))=0$. Hence $(Y,K)$ generates an infinite cyclic subgroup of $\coker \varphi $.
\end{proof}

We conclude this section with several remarks that suggest further avenues of research.

\begin{remark} \label{rmk:summand}
The $\Z$ subgroup of $\CZhat/\CZ$ constructed in the proof of Theorem \ref{thm:main} \eqref{it:Zsubgroup} is not necessarily a direct summand. To produce a direct summand, one would need a surjective homomorphism $\psi \co \CZhat \to \Z$ (i.e., a homology concordance invariant that is additive under connected sum) which vanishes on $\varphi(\CZ)$. None of the aforementioned invariants have this property.
\end{remark}

\begin{remark} \label{rmk:epup}
We do not know whether there exists a knot $J \subset S^3$ with $\varep(S^3,J) =0$ but $\Upsilon_J
\not\equiv 0$, although this may well be due to the paucity of computed examples. If it can be
shown that no such knot exists, then the relationship between $\Upsilon$ and $\varep$ would provide
another obstruction to a class $(Y,K) \in \CZhat$ lying in $\im \varphi$, analogous to Corollary
\ref{cor:obstruction}. (For an example of a knot $J \subset S^3$ with $\Upsilon_J(t) \equiv 0$ but
$\varep(J) \neq 0$, see \cite{Homepup}.)
\end{remark}

\begin{remark} \label{rmk:pi1}
We may also consider a slight variant on the definition of $\CZhat$. Let $\CZhat'$ denote the subgroup of $\CZhat$ consisting of all pairs $(Y,K)$ such that $Y$ bounds a homology $4$-ball $X$ in which $K$ is freely nulhomotopic. (Two equivalent formulations of this condition are that $K$ represents the trivial element of $\pi_1(X)$, and that $K$ bounds an immersed disk in $X$.) The preceding discussion shows that $\CZ$ is contained in $\CZhat'$. Recent work of Daemi (which appeared after the original version of this article) shows that $\CZhat'$ and $\CZhat$ are not equal \cite[Remark 1.6]{Daemi}.  Namely, there exist classes in $\pi_1(\Sigma(2,3,5) \conn {-\Sigma(2,3,5)})$ which are not freely nullhomotopic in any homology ball. (See also \cite[Theorem 1.9]{Davis}.)  Note that in the context of $4$-manifold topology, $\CZhat'$ is arguably a more appropriate generalization of the concordance group than $\CZhat$, since it measures the failure of immersed disks to be replaced by embedded ones. However, in general it is difficult to determine membership in $\CZhat$ but not $\CZhat'$.

In any case, the pairs $(Y_j, K_j)$ with which we prove Proposition \ref{prop:theta-unbounded} all lie in $\CZhat'$, since the manifolds $Y_j$ bound contractible $4$-manifolds. It follows that $\CZhat'/\CZ$ is infinitely generated. We do not know whether the $(Y,K)$ from Proposition \ref{prop:epsilonzerotaunonzero} lies in $\CZhat'$, but it seems likely one can find an element of $\CZhat'$ satisfying the same conclusions.
\end{remark}

\begin{remark} \label{rmk:topological}
One can also consider the analogues of $\CZ$ and $\CZhat$ in the topological category. Namely, we
say knots $(Y_0,K_0)$ and $(Y_1,K_1)$ are \emph{topologically homology concordant} if they cobound
a locally flat embedded annulus in a topological homology cobordism between $Y_0$ and $Y_1$ (which
need not carry any smooth structure), and we let $\cC_{\Z,\mathrm{top}}$ and
$\widehat\cC_{\Z,\mathrm{top}}$ denote the corresponding concordance groups, as in
\cite{DavisRaySatellite}. Note that every homology $3$-sphere bounds a contractible topological
$4$-manifold, so there is no restriction on which pairs $(Y,K)$ are represented in
$\widehat\cC_{\Z,\mathrm{top}}$. We do not know whether the natural inclusion $\varphi\co
\cC_{\Z,\mathrm{top}} \to \widehat\cC_{\Z,\mathrm{top}}$ is an isomorphism.
\end{remark}

\subsection*{Organization}
In Sections \ref{sec:dinvts} and \ref{sec:knot-family}, we prove Propositions \ref{prop:thetabounds} and \ref{prop:theta-unbounded}, respectively, which together give the proof of Theorem \ref{thm:main} \eqref{it:infgen}. In Section \ref{sec:concinv}, we show that $\tau$, $\varep$, and $\Upsilon$ are invariants of homology concordance and prove several properties of these invariants. In Section~\ref{sec:surgery-cfk}, we review a generalization of Ozsv\'ath and Szab\'o's mapping cone surgery formula due to Hedden and the second author, which computes the knot Floer homology of the core
of a Dehn surgery.  In Section \ref{sec:mappingcone}, we use this formula to prove Proposition
\ref{prop:epsilonzerotaunonzero}. Unless otherwise specified, singular homology will be taken with
$\Z$-coefficients. When considering Heegaard Floer homology groups, we work over $\F=\Z/2\Z$.

\subsection*{Acknowledgements}
This project began while the first and third authors were members at the Institute for Advanced Study and the second was down the road at Princeton University. We are grateful to both institutions for the ideal working environment. We thank John Etnyre, Stefan Friedl, Robert Lipshitz, Dan Margalit, Yi Ni, Katherine Raoux, Danny Ruberman, and Eylem Yildiz for helpful conversations, as well as Andr\'e Haefliger for sharing a copy of the unpublished paper \cite{HaefligerLissage2}. We are also grateful to the referees for their insightful and patient suggestions through several rounds of editing, which have greatly improved the quality of this work. 

\section{$d$-invariants and concordance}\label{sec:dinvts}

In this section, we will study the invariants $N_Y$ and $\theta(Y,K)$ defined in the introduction,
culminating in the proof of Proposition \ref{prop:thetabounds}.

As noted above, for any knot $K$ in a homology sphere $Y$, and any $n \in \Z$, the integer
$d(Y_{1/n}(K))$ is an invariant of the homology concordance class of $K$. When $Y = S^3$, Ni and Wu
showed that the $d$-invariants of all Dehn surgeries on $K$ can be recovered
from the knot Floer complex of $K$ \cite[Proposition 1.6]{NiWu}. In particular, for $1/n$
surgeries, the $d$-invariant is determined by the invariant $V_0(K)$, which is a nonnegative integer defined in \cite[Section 2.2]{NiWu}. (See also \cite{RasmussenThesis}.)

\begin{theorem}[Ni--Wu \cite{NiWu}]\label{thm:signs}
Let $K$ be a knot in $S^3$.  Let $d_+$ and $d_-$ be the $d$-invariants of $+1$- and $-1$-surgery on $K$ respectively.  Then,
\[
d(S^3_{1/n}(K)) = \begin{cases}  d_- & \text{if } n < 0 \\  0 & \text{if } n = 0 \\ d_+ & \text{if } n >0. \end{cases}
\]
Moreover, $d_+ = -2V_0(K)$ and $d_- = 2V_0(\overline K)$, where $\overline{K}$ denotes the mirror of $K$.
\end{theorem}

The proof of Theorem \ref{thm:signs} breaks down when the ambient manifold $Y$ has nontrivial reduced Heegaard Floer homology, but it can be modified to give bounds on $d(Y_{1/n}(K))$. We will focus only on the case where $n>0$. As in the introduction, define
\begin{align*}
\theta(Y,K) & = \max_{m,n>0} \abs{d(Y_{1/m}(K)) - d(Y_{1/n}(K))} \\
N_Y &= \min \{n \geq 0 \mid U^n \cdot \HFred(Y) = 0 \}.
\end{align*}
The following lemma establishes that $\theta(Y,K)$ is finite and proves Proposition~\ref{prop:thetabounds} from the introduction.


\begin{lemma}\label{lem:thetabounds}
$  $
\begin{enumerate}
\item \label{it:theta} For any knot $K$ in a homology sphere $Y$, we have $\theta(Y,K) \leq 2N_Y$; thus, $\theta(Y,K)$ is well-defined.
\item \label{it:thetaconc} If $(Y,K) \sim (Y',K')$, then $\theta(Y,K) = \theta(Y',K')$.
\end{enumerate}
\end{lemma}

\begin{proof}
To prove \eqref{it:theta}, we will adapt the proof of Ni and Wu's formula for $d$ invariants of surgeries \cite [Proposition 1.6] {NiWu} to the case of a knot in an arbitrary homology sphere $Y$ rather than just $S^3$. This proof uses Ozsv\'ath and Szab\'o's mapping cone formula for rational surgeries \cite{OSrational}. We will revisit the mapping cone formula for $+1$ surgery in Section \ref{sec:surgery-cfk}; for now, we simply point out three key properties:
\begin{enumerate}[label=(\alph*)]
\item\label{cone:triangle} For any $n \neq 0$, the Heegaard Floer homology of $1/n$-surgery fits into an exact triangle
\[
\begin{tikzcd}[column sep=small]
H_*(\bbA) \arrow{rr}{\Df_{1/n}} & & H_*(\bbB) \arrow{dl} \\
& \HFp(Y_{1/n}(K)) \arrow{ul} &
\end{tikzcd}
\]
where $H_*(\bbB)$ is a sum of infinitely many copies of $\HFp(Y)$ and $H_*(\bbA)$ is an infinite sum of the Heegaard Floer homologies of large surgeries on $K$ in varying spin$^c$ structures.
\item\label{cone:surjective} When $n$ is positive, the image of $\Df_{1/n}$ contains $U^N H_*(\bbB)$ for $N \gg 0$.
\item\label{cone:grading} When $n$ is positive, the minimum grading of an element in $\ker \Df_{1/n}$ in the image of $U^N$ for $N \gg 0$ is given by $d(Y)-2V_0(K)$.
\end{enumerate}
Property~\ref{cone:triangle} is simply \cite[Equation (1)]{NiWu}; the other two can be extracted from Ni and Wu's work as follows. The proof of \cite[Proposition 1.6]{NiWu} relies on four key lemmas: Lemmas 2.4, 2.7, 2.8, and 2.9.  The first two lemmas apply verbatim for knots in arbitrary homology spheres, while Lemma 2.8 now takes the form of \ref{cone:surjective} instead of a surjection onto all of $H_*(\bbB)$.  The proof of Lemma 2.9 only shows that the elements in $\ker \Df_{1/n}$ in the image of $U^N$ for $N \gg 0$ can alternatively be computed through the kernel of the restriction of $\Df_{1/n}$ to $U^{N'} H_*(\bbA)$ for ${N'}  \gg 0$.  The proof of Proposition 1.6 identifies the minimal grading of elements in $H_*(\bbA)$ that are simultaneously in the image of $U^N$ for $N \gg 0$ and in the kernel of $\Df_{1/n}$ restricted to $U^{N'} H_*(\bbA)$ for $N' \gg 0$.  In particular, their argument computes this grading to be $d(Y) - 2V_0(K)$, as in \ref{cone:grading}.

With this, we can now quickly give the proof.   The exact triangle in Property~\ref{cone:triangle} yields the short exact sequence
\[ 0 \rightarrow \coker \Df_{1/n} \rightarrow \HFp(Y_{1/n}(K)) \rightarrow \ker \Df_{1/n} \rightarrow 0 .\]
From Property~\ref{cone:surjective}, we deduce that $U^{N_Y} \coker \Df_{1/n} = 0$ when $n > 0$.

Thus, when $n > 0$, by Property~\ref{cone:grading}, the lowest grading of an element in $\HFp(Y_{1/n}(K))$ in the image of $U^N$ for all $N \gg 0$ is at least $d(Y)-2V_0(K)-2N_Y$.  This implies that
\[ d(Y)-2V_0(K)-2N_Y \leq d(Y_{1/n}(K)) \leq d(Y)-2V_0(K)\]
for all integers $n>0$. (For the second inequality, recall that $d(Y_{1/n}(K))$ is equal to the minimum grading of an element in $U^N \HFp(Y_{1/n}(K))$ for $N \gg 0$.) In particular, $|d(Y_{1/m}(K))-d(Y_{1/n}(K))| \leq 2N_Y$ for all positive integers $m, n$, as desired.

Finally, to prove \eqref{it:thetaconc}, we observe that if $(Y,K) \sim (Y',K')$, then $Y_{1/n}(K)$ is homology cobordant to $Y'_{1/n}(K')$ for all $n$, and in particular $d(Y_{1/n}(K))=d(Y'_{1/n}(K'))$. Therefore, $\theta(Y,K) = \theta(Y',K')$.
\end{proof}

\section{Infinite generation} \label{sec:knot-family}

As noted in the introduction, to complete the proof that $\CZhat/\CZ$ is infinitely generated (Theorem~\ref{thm:main} \eqref{it:infgen}), we must simply find a family of pairs $(Y_j,K_j) \in \CZhat$ for which $\theta$ is unbounded as $j \to \infty$. We will arrange that each $Y_j$ bounds
a contractible manifold, which guarantees that these elements are actually contained in the subgroup $\CZhat'$ (see Remark \ref{rmk:pi1}).

For $j \ge 1$, let $L_j \subset S^3$ denote the $2$-bridge link shown in either part of Figure \ref{fig:Lj}. (We leave to the reader to check that these diagrams are equivalent.) Denote the components of $L_j$ by $L_j^1$ and $L_j^2$, and observe from Figure \ref{subfig:Lj-symmetric} that there is an involution of $S^3$ which exchanges the two components. Note that $\lk(L_j^1, L_j^2) = \pm 1$ depending on the choice of orientation.

\begin{figure}
\subfigure[]{
\labellist
 \pinlabel $j$ at 36 38
 \pinlabel $j+1$ at 105 38
 \pinlabel $L_j^1$ [r] at 8 58
 \pinlabel $L_j^2$ [r] at 8 18
\endlabellist
\includegraphics{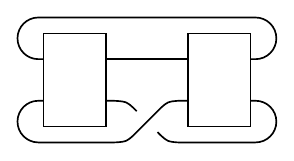} \label{subfig:Lj}}
\quad
\subfigure[]{
 \labellist
 \pinlabel $L_j^1$ [r] at 8 58
 \pinlabel $L_j^2$ [r] at 8 18
 \pinlabel $j+1$ at 222 38
 \endlabellist
\includegraphics{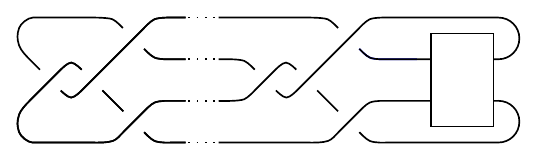} \label{subfig:Lj-symmetric}}
\caption{Two diagrams for the $2$-bridge link $L_j$. Each box indicates the number of right-handed whole twists. In the second figure, the repeated group of four crossings occurs $j$ times.} \label{fig:Lj}
\end{figure}

Let $J_j$ be the knot in $S^3$ obtained from $L_j^2$ by performing $+1$ surgery on $L_j^1$ and
blowing down. It follows that $S^3_n(J_j) = S^3_{1,n+1}(L_j)$ for any integer $n$. In particular, let
$Y_j = S^3_{-1}(J_j) = S^3_{1,0}(L_j)$, and let $K_j \subset Y_j$ denote the core circle of the surgery solid torus from the surgery on $J_j$. By zero-dot replacement (see, e.g. \cite[Section 5.4]{GompfStipsicz}), we see that $Y_j$ bounds a Mazur-type contractible $4$-manifold built from a single $1$-handle and a single $2$-handle. (These Mazur manifolds have also been studied by Akbulut and Karakurt \cite{AkbulutKarkurt}; using their notation, $Y_j = \partial (W_j(1))$.) Let $K_j \subset Y_j$ denote the core circle of the surgery solid torus. The main result of this section is as follows:

\begin{theorem} \label{thm:theta(YjKj)}
For each $j \ge 1$, we have
\[
\theta(Y_j,K_j) = 2 \left\lceil \frac{j}{2} \right\rceil.
\]
In particular, $\theta(Y_j,K_j)$ is unbounded as $j \to \infty$.
\end{theorem}

For each integer $n$, note that $(Y_j)_{1/n}(K_j) = S^3_{1/(n-1)}(J_j)$. (Compare Example \ref{ex:T23}.) In particular, $(Y_j)_1(K_j) = S^3$, so $d((Y_j)_1(K_j)) = 0$, while for $n>1$, Theorem \ref{thm:signs} gives
\[
d((Y_j)_{1/n}(K_j)) = d(S^3_{1/(n-1)}(J_j)) = -2V_0(J_j).
\]
Thus, Theorem \ref{thm:theta(YjKj)} will follow immediately from the following statement:

\begin{proposition} \label{prop:V0(Jj)}
For each $j \ge 1$, we have
\[
V_0(J_j)
= \left\lceil \frac{j}{2} \right\rceil.
\]
\end{proposition}

Rather than computing $V_0(J_j)$ directly from the complexes $\CFKi(S^3,J_j)$, which are difficult
to determine explicitly, we will compute $V_0(J_j)$ indirectly as follows. Let $M_j =
S^3_{2j+1}(J_j)$, and let $\spincs_0$ denote the unique self-conjugate Spin$^c$ structure on $M_j$.
We first show that $M_j$ is homeomorphic to a certain Seifert fibered space that bounds a
positive-definite plumbing. Using Ozsv\'ath and Szab\'o's algorithm from \cite{OSPlumbed}, we
compute $d(M_j, \spincs_0)$. This computation together with a formula of Ni and Wu \cite[Proposition 1.6]{NiWu} (similar to Theorem \ref{thm:signs} above) will determine $V_0(J_j)$.

\begin{lemma} \label{lem:Mj-seifert}
For each $j \ge 1$, the manifold $M_j = S^3_{2j+1}(J_j)$ is a Seifert fibered space of type $(2;
(2j+1,2j), (2j+1,j+1), (2j+3,j+2))$.\footnote{Many notational conventions for Seifert fibered
spaces exist in the literature; ours follows that of Saveliev \cite[Section
1.1.4]{Saveliev3Spheres}.}
\end{lemma}

\begin{figure}
\subfigure[]{
\labellist
 \small
 \pinlabel $j$ at 36 38
 \pinlabel $j+1$ at 106 38
 \pinlabel $2j+2$ [t] at 106 8
 \pinlabel $1$ [b] at 106 68
 \pinlabel $L_j^1$ [r] at 8 58
 \pinlabel $L_j^2$ [r] at 8 18
\endlabellist
\includegraphics{Lj} \label{fig:Lj-surgery}}
$\xrightarrow{\text{slide}}$
\subfigure[]{
\labellist
 \small
 \pinlabel $j+1$ at 86 41
 \pinlabel $2j+2$ [t] at 86 11
 \pinlabel $2j+5$ [b] at 86 71
\endlabellist
\includegraphics{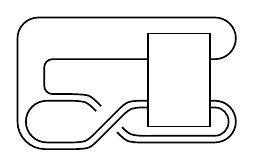} \label{fig:Lj-slide}}
$\xrightarrow{\cong}$
\subfigure[]{
\labellist
 \small
 \pinlabel $j+1$ at 45 45
 \pinlabel $2j+5$ [t] at 31 8
 \pinlabel $2j+2$ at 31 90
\endlabellist
\includegraphics{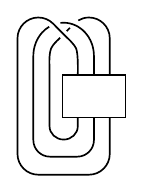} \label{fig:Lj-slide-isotopy}}
$\xrightarrow{\text{blow up}}$
\subfigure[]{
\labellist
 \small
 \pinlabel {$j$} [b] at 31 66
 \pinlabel $-2j+1$ [t] at 94 8
 \pinlabel $-1$ [t] at 64 42
 \pinlabel $-2$ [t] at 47 39
 \pinlabel $-2$ [t] at 29 39
 \pinlabel $-2$ [t] at 12 39
 \pinlabel $j+1$ at 95 89
\endlabellist
\includegraphics{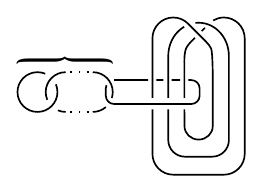} \label{fig:Lj-blowup}}
$\xrightarrow{\cong}$
\subfigure[]{
\labellist
 \small
 \pinlabel $-2j+1$ [tl] at 76 18
 \pinlabel $-1$ [t] at 61 8
 \pinlabel $-2$ [t] at 13 28
 \pinlabel $-2$ [t] at 28 28
 \pinlabel $-2$ [t] at 44 28
 \tiny
 \pinlabel $j+1$ [tl] at 77 35
\endlabellist
\includegraphics{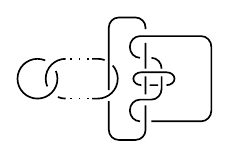} \label{fig:Lj-blowup-isotopy}}
$\xrightarrow{\text{blow up}}$
\subfigure[]{
\labellist
 \small
 \pinlabel $-2$ [b] at 15 66
 \pinlabel $-2$ [b] at 33 66
 \pinlabel $-2$ [b] at 51 66
 \pinlabel $-3$ [b] at 68 66
 \pinlabel $-1$ [b] at 91 66
 \pinlabel $-2j-1$ [b] at 118 65
 \pinlabel $-2$ [r] at 81 34
 \pinlabel $j$ [r] at 81 18
\endlabellist
\includegraphics{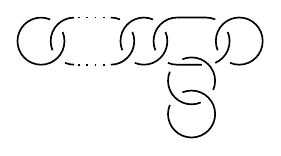} \label{fig:plumbing}}
\caption{Kirby moves showing that $M_j$ bounds a plumbing of spheres.} \label{fig:kirby}
\end{figure}

\begin{figure}
\subfigure[]{
\labellist
\pinlabel $-2j-1$ [bl] at 85 51
\pinlabel $-\frac{2j+1}{j}$ [tl] at 49 12
\pinlabel $-\frac{2j+3}{j+1}$ [br] at 14 51
\pinlabel $-1$ [b] at 49 57
\endlabellist
\includegraphics{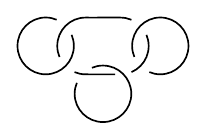}
\label{subfig:seifert1}
}
\hspace{0.5in}
\subfigure[]{
\labellist
\pinlabel $\frac{2j+1}{2j}$ [bl] at 84 51
\pinlabel $\frac{2j+1}{j+1}$ [tl] at 55 12
\pinlabel $\frac{2j+3}{j+2}$ [br] at 14 51
\pinlabel $2$ [b] at 49 57
\endlabellist
\includegraphics{seifert}
\label{subfig:seifert2}
}
\caption{Rational surgery descriptions of $M_j$ as a Seifert manifold.}
\label{fig:seifert}
\end{figure}

\begin{proof}
We begin with the surgery diagram for $M_j$ shown in Figure \ref{fig:Lj-surgery}, with the two 
components
of $L_j$ given framings $1$ and $2j+2$. Sliding $L_j^1$ over $L_j^2$ produces the diagram in Figure \ref{fig:Lj-slide}, which is then isotopic to Figure \ref{fig:Lj-slide-isotopy}. A series of $j+1$ blowups then produces Figure \ref{fig:Lj-blowup}, which is isotopic to Figure \ref{fig:Lj-blowup-isotopy}. Two more blowups produce Figure \ref{fig:plumbing}, which can be recognized as a plumbing of $2$-spheres. (Since $j>0$, this plumbing is indefinite.) By a sequence of slam dunks, we obtain the rational Dehn surgery diagram in Figure \ref{subfig:seifert1}, which represents the Seifert manifold \[M(-1;
(2j+1,-1), (2j+1,-j), (2j+3,-j-1).\] By applying a Rolfsen twist (see, e.g., \cite[Section 5.3]{GompfStipsicz}) to each of the outer components
of the diagram, we see that $M_j$ can also be described as \[M(2; (2j+1,2j), (2j+1,j+1),
(2j+3,j+2))\] (Figure \ref{subfig:seifert2}), as required.
\end{proof}

\begin{lemma}
The $d$-invariant of $M_j$ in the self-conjugate Spin$^c$ structure $\spincs_0$ is given by
\[
d(M_j, \spincs_0) =
\begin{cases}
-\frac{j}{2} -1 & j \text{ odd} \\
-\frac{j}{2} & j \text{ even}.
\end{cases}
\]
\end{lemma}

\begin{figure}
\labellist
\pinlabel $j+2$ [b] at 11 81
\pinlabel $2$ [b] at 44 83
\pinlabel $2$ [b] at 76 83
\pinlabel $2$ [b] at 109 83
\pinlabel $2$ [b] at 173 83
\pinlabel $2$ [b] at 205 83
\pinlabel $2$ [r] at 72 47
\pinlabel $j+1$ [r] at 72 11
\pinlabel $v_1$ [t] at 11 76
\pinlabel $v_2$ [t] at 44 76
\pinlabel $v_3$ [tl] at 80 75
\pinlabel $v_6$ [t] at 109 76
\pinlabel $v_{2j+4}$ [t] at 173 76
\pinlabel $v_{2j+5}$ [t] at 205 76
\pinlabel $v_4$ [l] at 80 47
\pinlabel $v_5$ [l] at 80 11
\endlabellist
\includegraphics{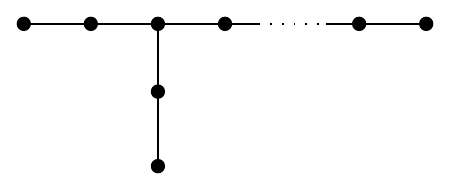}
\caption{The labeled graph $\Gamma_j$.}
\label{fig:plumbing-graph}
\end{figure}

\begin{proof}
Let $\Gamma_j$ be the labeled graph in Figure \ref{fig:plumbing-graph}. This graph has $2j+5$ vertices, which we label $v_1, \dots, v_{2j+5}$. The labels of the vertices are given by $m(v_1) = j+2$, $m(v_5) = j+1$, and $m(v_i) = 2$ for all other $i$. Let $X_j$ denote the plumbing of $2$-spheres specified by $\Gamma_j$. Using the rational surgery picture in Figure \ref{subfig:seifert2} together with the continued fraction expansions $\frac{2j+1}{j+1} = [2,j+1]$, $\frac{2j+3}{j+2} = [2,j+2]$, and $\frac{2j+1}{2j} = [\underbrace{2, \dots, 2}_{2j}]$, we see that $\partial X_j = M_j$. Let $A_j$ be the symmetric matrix associated to $\Gamma_j$, which presents the intersection form of $X_j$.  By \cite[Theorem 5.2]{NeumannRaymond}, since
\[
2 - \frac{2j}{2j+1} - \frac{j+1}{2j+1} - \frac{j+2}{2j+3} > 0,
\]
$A_j$ is positive-definite. Also, observe that $\Gamma_j$ has only one bad vertex ($v_3$), where bad here means that $m(v)$ is strictly less than the valence of $v$.

We recall a few basic facts from \cite{OSPlumbed} in order to compute the desired $d$-invariant. Let $V_j = (\Z^{2j+5}, A_j)$ be the lattice
associated with the graph $\Gamma_j$, and let $V_j^*$ be the dual lattice. Under the identification
of $V_j^*$ with $H^2(X_j)$, the first Chern classes of Spin$^c$ structures on $X_j$ correspond to
the set $\Char(V_j)$ of characteristic covectors in $V_j^*$, i.e., linear functions $\alpha \co V \to
\Z$ with the property that $\alpha(v_i) \equiv m(v_i) \pmod 2$. Identifying $V_j^*$ with
$\Z^{2j+5}$, the (rational) square of a covector $\alpha$ is given by $\alpha^2 = \alpha^T A_j^{-1} \alpha$.
Additionally, two covectors $\alpha, \alpha'$ restrict to the same Spin$^c$ structure on $M_j$ iff $(\alpha -
\alpha')/2 = A_j \x$ for some $\x \in \Z^n$; we denote these equivalence classes by $\Char(V_j,
\spincs)$ for $\spincs \in \Spin^c(M_j)$. In particular, $\alpha$ restricts to $\spincs_0$ iff $\alpha = A_j \x$ for some $\x \in \Z^n$; this in turn implies that
\[
\alpha^2 = \alpha^T \x  = \x^T A_j \x.
\]
The main theorem of \cite{OSPlumbed} then says that for each Spin$^c$ structure $\spincs$ on $M_j$, we have
\[
d(M_j, \spincs) = \min_{\alpha \in \Char(X_j, \spincs)} \frac{\alpha^2 - b_2(X_j)}{4} = \min_{\alpha \in \Char(X_j, \spincs)} \frac{\alpha^2 - 2j -5}{4}.
\]
(The results of \cite{OSPlumbed} are stated for negative-definite plumbings; the version stated here follows from orientation reversal.)

When $j$ is odd, the covector $\alpha_0 = (1,0,\dots, 0)$ (i.e., $\alpha_0(v_1) = 1$ and $\alpha_0(v_i) = 0$ for $i > 1$) is characteristic and restricts to $\spincs_0$. To see the latter statement, the equation $A \x = \alpha_0$ can be written as
\begin{align*}
(j+2)x_1 + x_2 &= 1  & x_4 + (j+1)x_5 &= 0 \\
x_1 + 2x_2 + x_3 &= 0 & x_3 + 2x_6 + x_7 &= 0 \\
x_2 + 2x_3 + x_4 + x_6 &= 0 & x_i + 2x_{i+1} + x_{i+2} &= 0 \quad \text{for } i=6, \dots, 2j+3 \\
x_3 + 2x_4 + x_5 &= 0 & x_{2j+4} + 2x_{2j+5} &= 0
\end{align*}
which has the integral solution
\begin{align*}
x_1 &= 1 &  x_4 &= -j-1 \\
x_2 &= -j-1 & x_5 &= 1 \\
x_3 &= 2j + 1 & x_i &= (-1)^{i+1} (2j+6-i) \quad \text{for } i=6, \dots, 2j+5.
\end{align*}
Hence, we see that $\alpha_0^T A^{-1} \alpha_0 = 1$. For any other $\alpha \in \Char(V, \spincs_0)$, ${\alpha}^2$ must also be a positive integer, so $\alpha_0$ has minimal square. Thus, we deduce that
\[
d(M_j, \spincs_0) = \frac{1-2j-5}{4}
= -\frac{j}{2} - 1.
\]

Similarly, when $j$ is even, let $\alpha_0$ be the covector with $\alpha_0(v_5)=-1$, $\alpha_0(v_{2j+5})=2$, and $\alpha_0(v_i)=0$ for all other $i$. The equations are:
\begin{equation} \label{eq:solve-for-x-j-even}
\begin{aligned}
(j+2)x_1 + x_2 &= 0  & x_4 + (j+1)x_5 &= -1 \\
x_1 + 2x_2 + x_3 &= 0 & x_3 + 2x_6 + x_7 &= 0 \\
x_2 + 2x_3 + x_4 + x_6 &= 0 & x_i + 2x_{i+1} + x_{i+2} &= 0 \quad \text{for } i=6, \dots, 2j+3 \\
x_3 + 2x_4 + x_5 &= 0 & x_{2j+4} + 2x_{2j+5} &= 2
\end{aligned}
\end{equation}
The solution we get is:
\begin{align*}
x_1 &= 1 & x_4 &= -(j+2) \\
x_2 &= -(j+2) & x_5 &= 1 \\
x_3 &= 2j+3 & x_i &= (-1)^{i+1} (2j+8-i) \quad \text{for } i=6, \dots, 2j+5.
\end{align*}
Hence, $\alpha_0^2 = -1 \cdot 1 + 2 \cdot 3 = 5$. To check that $\alpha_0$ has minimal square, we claim that if $\alpha \in \Char(V, \spincs_0)$, then $\alpha^2 \equiv 5 \pmod 8$, and hence $\alpha^2 \ge 5$. To see this, suppose that $\alpha = A_j \x$ for $\x \in \Z^{2j+5}$. The equations \eqref{eq:solve-for-x-j-even} hold mod 2, which implies that $x_i \equiv i \pmod 2$ for $i=1, \dots, 2j+5$. We compute:
\begin{align*}
\alpha^2 &= \x^T A_j \x \\
&= (j+2) x_1^2 + (j+1)x_5^2 + 2 \sum_{i \ne 1, 5} x_i^2 + 2 \sum_{i \ne 5, 2j+5}  x_i x_{i+1} + 2x_3x_6 \\
&= (j+1) x_1^2 -x_3^2 + j x_5^2 + x_{2j+5}^2 + \sum_{i \ne 5, 2j+5} (x_i + x_{i+1})^2 + (x_3+x_6)^2 .
\end{align*}
(Here, the sums are taken over all $i=1, \dots, 2j+5$ except for the specified values.) Since any odd square is congruent to $1$ mod $8$, we see that
\begin{align*}
\alpha^2 &\equiv (j+1) -1 + j + 1 + 2j+4 \equiv 4j+5 \equiv 5 \pmod 8,
\end{align*}
as required. It thus follows that
\[
d(M_j, \spincs_0) = \frac{5-2j-5}{4} =  -\frac{j}{2}. \qedhere
\]
\end{proof}

\begin{proof}[Proof of Proposition \ref{prop:V0(Jj)}]
Ni and Wu's formula for $d$ invariants of integer surgeries \cite[Proposition 1.6]{NiWu} implies that
\[
d(M_j, \spincs_0) = \frac{j}{2} - 2V_0(J_j).
\]
Combining this with the results above, we see that
\[
V_0(J_j) =
\begin{cases}
\frac{j+1}{2} & j \text{ odd} \\
\frac{j}{2} & j \text{ even}
\end{cases}
\]
as required.
\end{proof}

We conclude this section by mentioning a few other properties of the knots $J_j \subset S^3$, which may be of interest in other contexts.

\begin{proposition} \label{prop:tau(Jj)}
For each $j \ge 1$, we have $\tau(J_j) = g_4(J_j) = g(J_j) = j$.
\end{proposition}

(Here $g$ denotes the Seifert genus and $g_4$ denotes the $4$-ball genus.)

\begin{figure}
\subfigure[]{
\labellist
\pinlabel $L_j^2$ [r] at 8 98
\pinlabel $L_j^1$ [tr] at 78 12
\endlabellist
\includegraphics{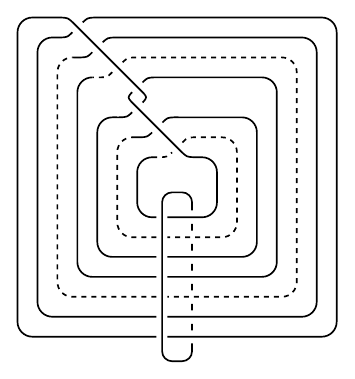}
\label{subfig:Lj-alternate}
}
\qquad
\subfigure[]{
\includegraphics{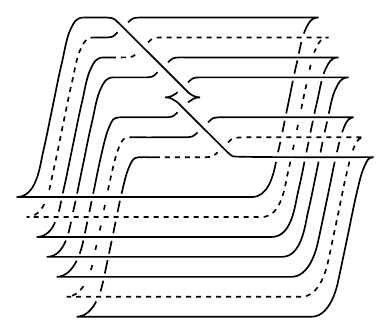}
\label{subfig:Jj-legendrian}
}
\caption{(a) Alternate picture of $L_j$. The dashed portion indicates winding multiple times. The curve $L_j^2$ passes through the disk spanned by $L_j^1$ $2j+1$ times, $j+1$ times in one direction and $j$ in the other. (b) Legendrian representative for $J_j$, obtained by blowing down $L_j^1$ with framing $+1$ from the previous figure.} \label{fig:Lj-alternate}
\end{figure}


\begin{proof}
First, we claim that $J_j$ has a genus-$j$ Seifert surface. To see this, orient the link $L_j$ so that the linking number of the two components is $-1$. In Figure \ref{subfig:Lj}, this means that the two strands passing through the box labeled $j$ have parallel orientation. Applying Seifert's algorithm to the diagram in Figure \ref{subfig:Lj} produces a Seifert surface $S$ for $L_j$ built out of $2j+3$ disks and $4j+3$ bands, which thus has genus equal to $j$. This surface induces the $+1$ framing on each component of $L_j$, so when we perform $+1$ surgery on $L_j^1$, we can cap off one boundary component of $S$ in the surgery to give a genus-$j$ Seifert surface for $J_j$. It follows  $\tau(J_j) \le g_4(J_j) \le g(J_j) \le j$.

To see the reverse inequalities, we adapt a contact-geometric argument given by Ray \cite{RayIterates}. The link $L_j$ can also be depicted as in Figure \ref{subfig:Lj-alternate}. Viewing the exterior of $L_j^1$ as a standard solid torus, $L_j^2$ is the same as the pattern knot $Q_j$ from \cite[Figure 9]{RayIterates}. Blowing down $L_j^1$ (with $+1$ framing) inserts a full negative twist in the $2j+1$ strands of $L_j^2$ that pass through it, producing $J_j$. This knot has a Legendrian representative $\mathcal{J}_j$ given by the front projection in Figure \ref{subfig:Jj-legendrian}, from which it is easy to compute that $\operatorname{tb}(\mathcal{J}_j) = 2j-1$ and $\operatorname{rot}(\mathcal{J}_j)=0$.
By Plamenevskaya's inequality \cite{PlamenevskayaBounds}, we deduce that $2j-1 \le 2\tau(J_j)-1$, hence $j \le \tau(J_j)$, as required.
\end{proof}

\begin{remark} \label{rmk:CFK(Jj)}
Using Figure \ref{fig:Lj-alternate}, it is not hard to show that each knot $J_j$ is a $(1,1)$-knot, which implies that the knot Floer complex $\CFKi(S^3,J_j)$ can be computed explicitly from a genus-$1$ Heegaard diagram. However, the number of generators of this complex grows quadratically as a function of $j$, making a general description difficult.
\end{remark}

\section{Concordance invariants from knot Floer homology}
\label{sec:concinv} In this section, we discuss concordance invariants coming from the Heegaard
Floer homology package. The main goal is to show that the basic properties of certain concordance
invariants ($\tau$, $\nu$, $\nu'$, $\varep$, and $\Upsilon$), which were originally only stated for
knots in $S^3$, hold in the more general setting of homology concordance of knots in homology
spheres.

We assume that the reader is familiar with the knot Floer complex, defined by Ozsv\'ath and Szab\'o in \cite{OSknots}. We use the notation of \cite[Section 2.2]{Homsurvey}. That is, given a knot $K$ in an integer homology sphere $Y$, we let $C=\CFKi(Y,K)$, which, upon choosing a filtered basis, decomposes as a direct sum $C = \bigoplus_{i,j\in \Z} C(i,j)$, such that
\[\partial(C(i,j)) \subset \bigoplus_{\substack{i' \le i \\ j' \le j}} C(i',j').\]
By \cite[Lemma 4.5]{RasmussenThesis}, we will assume throughout that $C$ is \emph{reduced}, i.e., that every term in the differential strictly lowers either $i$ or $j$.

For any set $X \subset \Z^2$ which is convex with respect to the product partial order on $\Z^2$ (i.e., if $a < b <c$ and $a,c \in X$, then $b\in X$), let $CX = \bigoplus_{(i,j) \in X} C(i,j)$, which is naturally a subquotient complex of $C$.

The key ingredient to extend the various Heegaard Floer concordance invariants for knots in $S^3$ to homology concordance invariants of knots in arbitrary homology spheres will come from a result of Zemke.

\begin{proposition}[{\cite[Theorem A]{Zemkelinkcobord}}]\label{prop:Zemke}
If $(Y_1, K_2)$ and $(Y_2, K_2)$ are homology concordant, then there exist filtered, grading-preserving $\F[U]$-equivariant chain maps
\[ F \co \CFKi(Y_1, K_1) \rightarrow \CFKi(Y_2, K_2) \] and \[ G \co \CFKi(Y_2, K_2) \rightarrow \CFKi(Y_1, K_1) \]
such that $F$ and $G$ induce isomorphisms on homology.
\end{proposition}

We now give the definitions of the concordance invariants we are interested in.  For $t \in [0, 2]$, $s \in \R$, and $C=\CFKi(Y, K)$, let
\[ C^t_s(Y, K) = C\left\{ i, j \mid \left (1-\frac{t}{2} \right)  i + \frac{t}{2} j \leq s \right\}. \]
Consider the maps
\begin{align*}
	\iota_s &\co C \{ i=0, j \leq s \} \rightarrow C \{ i = 0 \}, \\
	v_s &\co C \{ \max(i, j-s)=0 \} \rightarrow C \{ i = 0 \}, \\
	v'_s &\co C\{ i = 0\} \rightarrow C\{ \min(i, j-s) = 0 \}, \\
	f^t_s &\co C^t_s(Y, K) \rightarrow C,
\end{align*}
where $\iota_s$ is inclusion, $v_s$ consists of quotienting by $C\{i < 0, j=s\}$ followed by inclusion, $v'_s$ consists of quotienting by $C\{ i=0,  j < s\}$ followed by inclusion, and $f^t_s$ is inclusion. Recall that $C\{i=0\} \simeq \CFh(Y)$ and $C\{i\geq 0\} \simeq \CF^+(Y)$. Also, let $\rho \co \CFh(Y) \rightarrow \CF^+(Y)$ denote inclusion.

\begin{definition} \label{def:taunu}
Let $K$ be a knot in an integer homology sphere $Y$. Define
\begin{align*}
	 \tau(Y,K) &= \min \{ s \mid \Im (\rho_* \circ \iota_{s*}) \cap U^N \HF^+(Y) \neq 0 \ \forall N \gg 0\} \\
	 \nu(Y,K) &= \min \{ s \mid \Im (\rho_* \circ v_{s*}) \cap U^N \HF^+(Y) \neq 0 \ \forall  N \gg 0\} \\
	  \nu'(Y,K) &= \max \{ s \mid v'_{s*}(x)\neq 0 \ \forall x \in \HFh(Y) \textup{ s.t. } \rho_*(x) \neq 0 \textup{ and } \\ & \qquad \qquad \qquad \qquad \qquad \qquad \qquad \qquad \rho_*(x) \in U^N \HF^+(Y) \ \forall N \gg 0  \} \\
	  \Upsilon_{Y,K}(t) &= -2 \min \{ s \mid \Im f^t_{s*} \textup{ contains a non-trivial
element in grading $d(Y)$} \}.
\end{align*}
\end{definition}

\begin{remark}
When $Y=S^3$, these definitions agree with the definitions of $\tau$, $\nu$, $\nu'$, and $\Upsilon$ in \cite[Section 1]{OS4ball}, \cite[Definition 9.1]{OSrational}, \cite[Definition 3.1]{Homcables}, and \cite[Definition 5.2]{LivingstonUpsilon} respectively. Indeed, when $Y=S^3$, we have $\Im (\rho_* \circ \iota_{s*}) \cap U^N \HF^+(S^3) \neq 0$ for all $N \gg 0$ if and only if $\iota_{s*}$ is surjective. Hence the  definition of $\tau$ above agrees with the usual definition of $\tau$ for knots in $S^3$; similar arguments apply for $\nu, \nu'$, and $\Upsilon$. In the definition of $\Upsilon$, our use of the element of $\HFi(Y)$ in grading $d(Y)$ (as opposed to grading $0$) guarantees that when $K$ is the unknot in $Y$, $\Upsilon_{Y,K} \equiv 0$.
\end{remark}

\begin{remark} \label{rmk:nu'}
Equivalently, we can define $\nu'(Y, K)$ to be the maximum $s$ such that the map induced by the composition
\[ C\{ i \leq 0\} \rightarrow C\{i=0\} \rightarrow C\{\min(i, j-s)=0\}\]
is non-trivial on every maximally graded non-$U$-torsion element in $\HFm(Y) \cong H_*(C\{i \leq 0\})$, where the first map is quotienting by $C\{i<0\}$ and the second is $v'_s$. This definition agrees with Definition \ref{def:taunu} by the arguments in the proof of \cite[Proposition 2.13]{OSS}. From this alternate definition, it follows that $\nu$ and $\nu'$ are dual to one another in the sense that $\nu(-Y, K) = -\nu'(Y, K)$.
\end{remark}

The invariant $\varep$ is defined via the relation between $\tau, \nu$, and $\nu'$. It is straightforward to verify (cf. \cite[Proposition 3.1]{OS4ball} and \cite[Equation 34]{OSrational}) that \[ \nu(Y,K) = \tau(Y,K) \textup{ or } \tau(Y,K)+1 \qquad \textup{ and } \qquad \nu'(Y,K) = \tau(Y,K)-1 \textup{ or } \tau(Y,K). \]
Just as in the case of knots in $S^3$, we have:

\begin{lemma}
Let $K$ be a knot in an integer homology sphere $Y$. The following three cases are exhaustive and mutually exclusive:
\begin{itemize}
	\item $\nu(Y,K)=\tau(Y,K)+1$ and $\nu'(Y,K)=\tau(Y,K)$,
	\item $\nu(Y,K)=\tau(Y,K)$ and $\nu'(Y,K)=\tau(Y,K)-1$,
	\item $\nu(Y,K)=\tau(Y,K)$ and $\nu'(Y,K)=\tau(Y,K)$.
\end{itemize}
\end{lemma}

\begin{proof}
The arguments in \cite[Section 3]{Homcables} still apply when the ambient manifold is an arbitrary integer homology sphere. For completeness, we sketch the argument here. Let $\tau=\tau(Y,K)$. Suppose that $\nu'(Y,K)=\tau-1$. Then there exists a cycle $x \in \CFh(Y)$ such that
\begin{enumerate}
	\item $\rho_*([x])$ is a non-zero element of $U^N \HF^+(Y) \neq 0$ for $N \gg 0$,
	\item there exists $y \in C\{ \min(i, j-\tau) = 0 \}$ with $\d y =v'_{\tau} (x)$, where $\d$ denotes the differential on $C\{ \min(i, j-\tau) = 0 \}$.
\end{enumerate}
We may assume that $y$ has non-trivial projection to $C\{ i > 0, j=\tau\}$. Then let $\overline{y}$ be the image of $y$ under projection to $C\{ i > 0, j=\tau\}$. Consider the projection
\[ p \co C\{j=\tau\} \rightarrow C\{ i > 0, j=\tau\}. \]
Choose $y' \in p^{-1}(\overline{y})$. Consider $z=\d^\horz y'$, where $\d^\horz$ denotes the differential on $C\{ j=\tau\}$. Then $z$ is a cycle in $C\{i \leq 0, j=\tau \}$. Consider the projection
\[ q \co C\{ \max(i,j-\tau)=0\} \rightarrow C\{i \leq 0, j=\tau\}.\]
There exists a cycle $z' \in q^{-1}(z)$ such that  $v_{\tau *}([z']) = [x]$, i.e., $\nu(Y,K)=\tau$.
\end{proof}

By the preceding lemma, the following is well-defined.

\begin{definition}
Let $K$ be a knot in an integer homology sphere $Y$. Define
\[	\varep(Y,K) =
		\begin{cases}
			-1 \quad & \text{if } \nu(Y,K)=\tau(Y,K)+1, \\
			1 \quad & \text{if } \nu'(Y,K)=\tau(Y,K)-1, \\
			0 \quad & \text{otherwise}.
		\end{cases}
\]
\end{definition}

Using Proposition~\ref{prop:Zemke}, we can give a uniform proof that all of the invariants defined so far are invariants of homology cobordism, proving a generalization of Theorem \ref{thm:concinv}.  This argument was known to Zemke, but we include it for completeness.
\begin{proposition}\label{prop:tauepsilonhomcob}
If $K_1 \subset Y_1$ and $K_2 \subset Y_2$ are concordant in a homology cobordism between $Y_1$ and $Y_2$, then
\begin{align*}
\tau(Y_1,K_1) &= \tau(Y_2,K_2) \\
\nu(Y_1,K_1) &= \nu(Y_2,K_2) \\
\nu'(Y_1,K_1) &= \nu'(Y_2,K_2) \\
\varep(Y_1,K_1) &= \varep(Y_2,K_2) \\
\Upsilon_{Y_1,K_1} &= \Upsilon_{Y_2,K_2}.
\end{align*}
\end{proposition}
\begin{proof}
Let $C_i = \CFKi(Y_i, K_i)$. By Proposition \ref{prop:Zemke}, there exist filtered, grading-preserving $\F[U]$-equivariant chain maps
\[ F \co C_1 \rightarrow C_2 \qquad \textup{ and } \qquad G \co C_2 \rightarrow C_1 \]
such that $F$ and $G$ induce isomorphisms on homology. Because these maps are filtered, they induce maps
\[ F \co C_1 X \rightarrow C_2 X \qquad \textup{ and } \qquad G \co C_2 X \rightarrow C_1 X\]
for any subset $X \subset \Z^2$ that is convex with respect to the product partial order on $\Z^2$.

Consider the following commutative diagram:
\[
\begin{CD}
	C_1\{i=0, j \leq s\} @> \rho \iota_s >> C_1\{i \geq 0\} \\
	@VV{F}V			@VV{F}V	  \\
	C_2\{i=0, j \leq s\} @> \rho \iota_s >> C_2\{i \geq 0\}.
\end{CD}
\]
Since $F_*$ commutes with the $U$-action and is an isomorphism on $H_*(C_i)$, it follows that $\tau(Y_1, K_1) \geq \tau(Y_2, K_2)$. By considering the analogous diagram with $G$, we obtain $\tau(Y_1, K_1) \leq \tau(Y_2, K_2)$. Hence $\tau(Y_1, K_1)=\tau(Y_2, K_2)$.

Similarly, the proof that $\nu(Y_1,K_1)=\nu(Y_2,K_2)$ follows from considering the commutative diagram
\[
\begin{CD}
	C_1\{ \max(i, j-s)=0 \} @> \rho v_s >> C_1\{i \geq 0\} \\
	@VV{F}V			@VV{F}V	  \\
	C_2\{ \max(i, j-s)=0 \}@> \rho v_s >> C_2\{i \geq 0\}.
\end{CD}
\]
An analogous diagram for $v'_s$ shows that $\nu'$, and hence $\varep$, are homology concordance invariants.

Finally, the proof that $\Upsilon_{Y_1, K_1} = \Upsilon_{Y_2, K_2}$ follows from considering the commutative diagram
\[
\begin{CD}
	C^t_s(Y_1, K_1)@> f^t_s>> C_1 \\
	@VV{F}V			@VV{F}V	  \\
	C^t_s(Y_2, K_2)@> f^t_s >> C_2,
\end{CD}
\]
and the analogous diagram with $G$.
\end{proof}

\begin{remark}
For completeness, we note that an alternate proof that $\tau(Y_1,K_1) = \tau(Y_2,K_2)$ follows from Raoux's work \cite[Corollary 5.4]{Raouxtau} (see also \cite{OS4ball}).  There it is shown that if $(Y,K)$ bounds $(W,\Sigma)$, where $W$ is a rational homology ball, then $| \tau(Y,K)| \leq g(\Sigma)$.  If $(Y_1, K_1)$ and $(Y_2,K_2)$ are homology concordant, then $Y_1 \# -Y_2$ bounds a rational homology ball in which $K_1 \conn {-K_2}$ bounds an embedded disk. Therefore $\tau(Y_1 \conn {-Y_2}, K_1 \conn {-K_2}) = 0$, which implies that $\tau(Y_1,K_1) = \tau(Y_2,K_2)$ by Proposition~\ref{prop:tauprops} below.
\end{remark}

\begin{remark}
The same arguments apply to prove that the invariants $V_i$ \cite{NiWu}, $\nu^+$ \cite{HomWu4Genus}, and $\nu_n$ \cite{TruongTruncated} can be appropriately generalized to give invariants of homology concordance for knots in arbitrary homology spheres.
\end{remark}

The next three propositions show that $\tau$, $\varep$, and $\Upsilon$ have the same additivity properties for knots in arbitrary homology spheres as they do for knots in $S^3$. (Note that all of these properties are invariants of a pair $(Y,K)$ up to diffeomorphism, so our use of connected sum notation is justified.)

\begin{proposition}[{\cite[Proposition 3.10]{Raouxtau}}]\label{prop:tauprops}
Let $K_1$ and $K_2$ be knots in integer homology spheres $Y_1$ and $Y_2$, respectively. Then
\begin{enumerate}
	\item $\tau(-Y,K) = -\tau(Y,K)$.
	\item $\tau (Y_1 \conn Y_2, K_1 \conn K_2) = \tau(Y_1,K_1) + \tau(Y_2,K_2)$.
\end{enumerate}
\end{proposition}

\begin{proposition}\label{prop:epsilonprops}
Let $K_1$ and $K_2$ be knots in integer homology spheres $Y_1$ and $Y_2$, respectively. Then
\begin{enumerate}
	\item \label{it:ep(-Y,K)} $\varep(-Y,K) = -\varep(Y,K)$.
	\item \label{it:epadd1} If $\varep(Y_1, K_1)=\varep(Y_2, K_2)$, then $\varep(Y_1 \conn Y_2, K_1 \conn K_2) = \varep(Y_1, K_1)$.
	\item \label{it:epadd2} If $\varep(Y_1, K_1)=0$, then $\varep(Y_1 \conn Y_2, K_1 \conn K_2) = \varep(Y_2, K_2)$.
\end{enumerate}
\end{proposition}

\begin{proof}
\eqref{it:ep(-Y,K)} This follows from the definition of $\varep$ together with Remark \ref{rmk:nu'}.

\eqref{it:epadd1} We shall only do the case where $\varep(Y_1, K_1)=\varep(Y_2, K_2) = 0$, which is the only case used in this paper; the remaining cases follow similarly. Let $\tau_k=\tau(Y_k, K_k)$ and $C_k=\CFKi(K_k)$ for $k=1,2$. Suppose that $\varep(Y_k,K_k)=0$. Then $\nu(Y_k,K_k)=\tau_k$ and $\nu'(Y_k,K_k)=\tau_k$. Since $\nu(Y_k,K_k)=\tau_k$, there exists a cycle $x_k$ in $C_k\{\max(i, j-\tau_k)=0\}$ such that $\rho_* \circ v_{\tau_k *}([x_k]) \in \Im U^N \subset \HF^+(Y_k)$ for $N \gg 0$ and $\rho_* \circ v_{\tau_k *}([x_k]) \neq 0$. In particular, the image of $x_k$ under the quotient from $C_k\{\max(i, j-\tau_k)=0\}$ to $C_k\{(0, \tau_k)\}$ is non-trivial.

Let $\tau=\tau(Y_1 \conn Y_2, K_1 \conn K_2)$, which equals $\tau_1+\tau_2$ by Proposition~\ref{prop:tauprops}, and $C=\CFKi(Y_1\conn Y_2, K_1 \conn K_2)$, which is isomorphic to $C_1\otimes C_2$ by \cite{OSknots}. Then $x_1 \otimes x_2$ is a cycle in $C\{ \max(i, j-\tau_1-\tau_2)=0\}$ and $\rho_* \circ v_{\tau*}([x_1\otimes x_2]) \in \Im U^n \subset \HF^+(Y_1 \# Y_2)$, i.e., $\nu(Y_1 \conn Y_2, K_1 \conn K_2) = \tau$. A similar argument shows that $\nu'(Y_1 \conn Y_2, K_1 \conn K_2) = \tau$. Thus, $\varep (Y_1\conn Y_2, K_1 \conn K_2) = 0$.

\eqref{it:epadd2} This is left as an exercise to the reader, following \cite[Proposition 3.6]{Homcables}. (The case where $\varep(K_2)=0$ is treated above.)
\end{proof}

\begin{proposition} \label{prop:upsilonprops}
Let $K_1$ and $K_2$ be knots in integer homology spheres $Y_1$ and $Y_2$, respectively. Then for each $t \in [0,2]$,
\begin{enumerate}
	\item $\Upsilon_{-Y,K}(t) = -\Upsilon_{Y,K}(t)$.
	\item $\Upsilon_{Y_1 \conn Y_2, K_1 \conn K_2}(t) = \Upsilon_{Y_1,K_1}(t) + \Upsilon_{Y_2,K_2}(t)$.
\end{enumerate}
\end{proposition}

\begin{proof}
Livingston's proof for knots in $S^3$ \cite[Theorem 6.2]{LivingstonUpsilon} carries through identically in arbitrarily homology spheres.
\end{proof}

However, as noted in the introduction, there is one property of $\tau$ and $\varep$ which does not generalize to arbitrary homology spheres; its failure to generalize is crucial to our proof of Theorem \ref{thm:main} \eqref{it:Zsubgroup}.

\begin{proposition}\label{prop:epsilonzerotauzero}
Suppose $K$ is a knot in $S^3$, or more generally in an integer homology sphere L-space $Y$. If $\varep(Y, K)=0$, then $\tau(Y,K)=0$.
\end{proposition}

\begin{proof}
This follows from the proof of \cite[Proposition 3.6(2)]{Homcables}; indeed, the only feature of the ambient manifold used in the proof is that it is an integer homology sphere L-space.
\end{proof}

\begin{proof}[Proof of Corollary \ref{cor:obstruction}]
The corollary follows from Propositions \ref{prop:tauepsilonhomcob} and \ref{prop:epsilonzerotauzero}.
\end{proof}

\begin{remark}
We have implicitly been working with \emph{oriented} knots throughout the section. By
\cite[Proposition 3.9]{OSknots}, note that $\CFKi(Y,K)$ is invariant under orientation
reversal for knots in homology spheres, since the unique spin$^c$ structure on $Y$ is necessarily
self-conjugate.  It follows that all the invariants discussed in this section are also invariant
under knot orientation reversal.
\end{remark}

\section{The filtered mapping cone formula} \label{sec:surgery-cfk}

Continuing with the notation from the previous section, let $C = \CFKi(Y,K)$ be the reduced, doubly-filtered knot Floer complex of a knot $K$ in a homology sphere $Y$. We will now briefly describe the surgery formula from \cite{HeddenLevine} for computing $\CFKi(Y_1(K), \tilde K)$, where $\tilde K$ is the core circle of the surgery solid torus in $1$-surgery on $K$. Essentially, this formula entails putting an extra filtration on the mapping cone complex for $\HFp(Y_1(K))$ given by Ozsv\'ath and Szab\'o \cite{OSinteger}.

To begin, there is a $U$-equivariant, grading-preserving chain homotopy equivalence $\Psi^\infty \co C \to C$, which restricts to a homotopy equivalence between the subcomplexes $C\{j \le s\}$ and $C \{i \le s\}$. This map $\Psi^\infty$ need not behave well with respect to the other filtration grading.

For each integer $s$, let $A_s^\infty$ and $B_s^\infty$ each denote a copy of the chain complex $C$, and write $A_s^\infty = \bigoplus_{i,j \in \Z} A_s^\infty(i,j)$ and $B_s^\infty = \bigoplus_{i,j \in \Z} B_s^\infty(i,j)$. We define a pair of $\Z$-filtrations $\II$ and $\JJ$ on each of these complexes as follows:
\begin{align*}
\II(A_s^\infty(i,j)) &= \max(i,j-s)  &  \JJ(A_s^\infty(i,j)) &= \max(i+s-1, j) \\
\II(B_s^\infty(i,j)) &= i & \JJ(B_s^\infty(i,j)) &= i+s-1.
\end{align*}
Let $A_s^-$ (resp. $B_s^-$) denote the subcomplex of $A_s^\infty$ (resp. $B_s^\infty$) with $\II<0$, let $A_s^+$ (resp. $B_s^+$) denote the quotient, and let $\hat A_s$ (resp. $\hat B_s$) be the subcomplex of $A_s^+$ (resp. $B_s^+$) with $\II=0$. The Maslov (homological) grading on each complex $A_s^\infty$ (resp.~$B_s^\infty$) is defined to be the Maslov grading on $C$, shifted up by $s(s-1)$ (resp.~$s(s-1)-1$). (The definitions of the $A_s^+$ and $B_s^+$ complexes agree with those in \cite{OSinteger}; the filtration $\JJ$ is introduced in \cite{HeddenLevine}.)

Define maps $v_s^\circ \co A_s^\circ \to B_s^\circ$ and $h_s^\circ \co A_s^\circ \to B_{s+1}^\circ$ as follows: $v_s^\infty$ is the identity map of $C$, and $h_s^\infty$ is $\Psi^\infty$ composed with multiplication by $U^s$. Each of these maps is filtered with respect to both $\II$ and $\JJ$. (For $v_s^\infty$, this is obvious; for $h_s^\infty$, it uses the filtration property of $\Psi^\infty$.) In particular, $v_s^\infty$ (resp. $h_s^\infty$) takes the subcomplex $A_s^-$ into $B_s^-$ (resp. $B_{s+1}^-$), and hence induces a map $v_s^+ \co A_s^+ \to B_s^+$ (resp. $h_s^+ \co A_s^+ \to B_{s+1}^+$), which agrees with the definition in \cite{OSinteger}. Moreover, each of $v_s^\infty$ and $h_s^\infty$ is homogeneous of degree $-1$ with respect to the (shifted) Maslov grading.

For any integers $a \le b$, let
\[
D^\infty_{a,b} \co \bigoplus_{s=a}^b A_s^\infty \to \bigoplus_{s=a+1}^b B_s^\infty
\]
be the map given by the sum of the maps $v_s^\infty \co A_s^\infty \to B_s^\infty$ ($s=a+1, \dots, b$) and $h_s^\infty \co A_s^\infty \to B_{s+1}^\infty$ ($s = a, \dots, b-1$), and let $\X^\infty_{a,b}$ denote the mapping cone of $D^\infty_{a,b}$. It is easy to see that $\II$ and $\JJ$ give $\X^\infty_{a,b}$ the structure of a doubly filtered chain complex with an action of $\F[U,U^{-1}]$. Moreover, it is shown in \cite[Lemma 3.2]{HeddenLevine} that for any $a \le 1-g$ and $b \ge g$ (where $g$ is the genus of $K$), the doubly-filtered chain homotopy type of $\X_{a,b}^\infty$ is independent of $a$ and $b$. Thus, we may define $\X^\infty = \X^\infty_{1-g,g}$. Denote the differential on $\X^\infty$ by $\d^\infty$.
%

The following theorem is the main result of \cite{HeddenLevine}:
\begin{theorem} \label{thm:surgery}
The chain complex $\X^\infty$ is filtered chain homotopy equivalent to $\CFKi(Y_1(K), \tilde K)$, where the filtrations $\II$ and $\JJ$ on $\X^\infty$ correspond to $i$ and $j$ on $\CFKi(Y_1(K), \tilde K)$.
\end{theorem}

Using the two filtration functions $\II,\JJ$, we may view the vector space $\X^\infty$ as a direct sum of pieces $\bigoplus_{i,j} \X(i,j)$. Given a filtered basis for $\X^\infty$, we may write $\d^\infty = \d + \d'$, where $\d$ consists of the terms that preserve both filtrations and $\d'$ consists of the terms that strictly drop at least one of them. Since the action of $U$ takes $\X(i,j)$ isomorphically to $\X(i-1,j-1)$, we can understand $\d$ by looking only at the summands with $i=0$. Each summand $\X(0,s)$ (with its internal differential $\d$) can be identified with $\CFKh(Y_1(K), \tilde K, s)$, the associated graded complex of $\CFKh(Y_1(K), \tilde K)$ in Alexander grading $s$. Hence $H_*(\X(0,s)) \cong \HFKh(Y_1(K), \tilde K, s)$. For $-g<s<g$, $\X(0,s)$ is easily described as the mapping cone
\begin{equation} \label{eq:hat-mapping-cone}
 A_s\{i \le 0, j=s\} \oplus A_{s+1}\{i = 0, j \le s\} \xrightarrow {(h_s, \ v_{s+1})} B_{s+1}\{i=0\},
\end{equation}
while $\X(0,-g) = A_{-g+1}(0,-g)$ and $\X(0,g) = A_g(0,g)$ (with vanishing differential).

Moreover, the $\II=0$ subquotient of $\X^\infty$ (whose homology is $\HFh(Y_1(K))$) is also isomorphic to $\bigoplus_{s=-g}^g \X(0,s)$ as a group. The higher differentials on $\CFKh(Y_1(K), \tilde K)$ (that is, the ones that decrease the $\JJ$ grading) are given completely by the part of the internal differential on $\hat A_s$ that takes $A_s(0,s) \subset A_s\{i\le
0, j = s\} \subset \X(0,s)$ into $A_s\{0, j \le s-1\} \subset \X(0,s-1)$. (While all of these differentials decrease $\JJ$ by $1$, terms which shift $\JJ$ by more than $1$ may nevertheless arise after passing to a
reduced model for $\CFKh(Y_1(K), \tilde K)$.) This description essentially agrees with the surgery formula for $\CFKh(Y_1(K), \tilde K)$ described by Eftekhary \cite{EftekharyIncompressible}, modulo some differences in conventions and notation. (However, Eftekhary's work does not describe the full doubly filtered complex $\CFKi(Y_1(K), \tilde K)$.)

Typically, one wants to obtain a \emph{reduced} model for $\CFK^\infty$ (i.e., one in which every term in the differential strictly lowers at least one of the filtrations), which makes it easy to read off invariants such as $\tau$, $\varep$, and $\Upsilon$ as in the previous section. We may pass from $\X^\infty$ to a reduced via the following ``cancellation'' procedure (see, e.g., \cite[Proposition 11.57]{LOT}). In each summand $\X(0,s)$, choose a basis $\{y_i\}$ for $\Im(\d)$, and choose elements $x_i \in \X(0,s)$ such that $\d(x_i) = y_i$. Then $\partial^\infty(x_i) = y_i + {}$ terms in lower filtration levels. The subcomplex of $\X^\infty$ spanned (over $\F[U,U^{-1}]$) by all the $\{x_i, \d^\infty(x_i)\}$ is acyclic, and the quotient $Q$ of $\X^\infty$ by this subcomplex is reduced and is filtered homotopy equivalent to the original complex $\X^\infty$. The generators for $Q$ (over $\F[U,U^{-1}]$) are naturally in bijection with the generators (over $\F$) of $\HFKh(Y_1(K), \tilde K)$, and the differential is induced from the terms in $\d'$ which are not-filtration preserving. As a practical matter, it is thus useful to begin by using the individual complexes $\X(0,s)$ to compute $\HFKh(Y_1(K), \tilde K)$. We shall carry out a computation using this strategy in Section \ref{sec:mappingcone}.

\section{The knot Floer complex for the core of surgery on a cable of the trefoil}\label{sec:mappingcone}

Throughout this section, let $T=T_{2,-3;2,3}$, the $(2,3)$-cable of the left-handed trefoil. (Here,
$2$ denotes the longitudinal winding and $3$ denotes the meridional winding.) Note that the genus of $T$ is 3; indeed, a minimal genus Seifert surface can be built from two parallel copies of Seifert surfaces for $T_{2, -3}$ together with three twisted bands between them. Let $Z=S^3_1(T)$, and
let $\T \subset Z$ denote the knot obtained as the core of the surgery. In this section, we will
use the mapping cone formula from the previous section to compute $\CFKi(Z, \T)$ and the associated
invariants $\tau$, $\varep$, and $\Upsilon$, and use this computation to prove Proposition
\ref{prop:epsilonzerotaunonzero}, which gives an infinite cyclic subgroup of $\CZhat/\CZ$. The main
technical result of this section is the following proposition:

\begin{proposition} \label{prop:cfkcore}
Consider $T=T_{2,-3;2,3}$ as above. Then $\CFKi(Z, \T)$ is generated over $\F[U,U^{-1}]$ by
generators $A, \dots, M$, with differential as shown in Figure \ref{fig:CFKinftycore}. (The Maslov gradings of the generators are given in Table \ref{tab:CFKinftycore}, below.)
\end{proposition}

\begin{figure}
\begin{tikzpicture}

	\begin{scope}[thin, black!20!white]
		\draw [<->] (-3, 0.5) -- (4, 0.5);
		\draw [<->] (0.5, -3) -- (0.5, 4);
	\end{scope}
	\draw[step=1, black!50!white, very thin] (-2.9, -2.9) grid (3.9, 3.9);
	
	\filldraw (-1.5, 1.5) circle (2pt) node[] (A){};
	\filldraw (-1.5, 0.5) circle (2pt) node[] (B){};
	\filldraw (-0.65, 0.35) circle (2pt) node[] (D){};	
	\filldraw (-1.5, -0.5) circle (2pt) node[] (C){};
	\filldraw (0.55, 1.5) circle (2pt) node[] (F){};	
	\filldraw (-0.45, 0.55) circle (2pt) node[] (E){};	
	\filldraw (0.55, 0.55) circle (2pt) node[] (G){};
	\filldraw (0.55,-0.45) circle (2pt) node[] (J){};
	\filldraw (1.5, 0.55) circle (2pt) node[] (K){};
	\filldraw (0.35, -0.65) circle (2pt) node[] (I){};
	\filldraw (-0.5, -1.5) circle (2pt) node[] (H){};
	\filldraw (1.5, -1.5) circle (2pt) node[] (M){};
	\filldraw (0.5, -1.5) circle (2pt) node[] (L){};

	\draw [very thick, ->] (A) -- (B);
	\draw [very thick, ->] (D) -- (C);
	\draw [very thick, ->] (F) -- (E);
	\draw [very thick, ->] (G) -- (E);
	\draw [very thick, ->] (G) -- (J);
	\draw [very thick, ->] (I) -- (H);
	\draw [very thick, ->] (K) -- (J);
	\draw [very thick, ->] (M) -- (L);

	\node [above] at (A) {$U^2A$};
	\node [left] at (B) {$U^2B$};
	\node [left] at (C) {$U^2C$};
	\node [below,xshift=5] at (D) {$UD$};
	\node [above,xshift=-5] at (E) {$UE$};
	\node [above] at (F) {$F$};
	\node [above] at (G) {$G$};
	\node [left] at (H) {$UH$};
	\node [left] at (I) {$I$};
	\node [right,xshift=1] at (J) {$J$};
	\node [right] at (K) {$U^{-1}K$};
	\node [below] at (L) {$L$};
	\node [right] at (M) {$U^{-1}M$};

\end{tikzpicture}
\caption{The reduced model for $\CFKi(Z,\T)$, drawn in the $(i,j)$-plane. }
\label{fig:CFKinftycore}
\end{figure}

Before proving this result, we show how it implies Proposition \ref{prop:epsilonzerotaunonzero}.

\begin{corollary} \label{cor:tauepsiloncore}
The knot $\T \subset Z$ satisfies:
\begin{align*}
\tau(Z,\T) &= -1 \\
\varep(Z,\T) &= 0 \\
\Upsilon_{Z,\T}(t) &=
\begin{cases}
t  & 0 \le t \le 1 \\
2-t & 1 \le t \le 2.
\end{cases}
\end{align*}
\end{corollary}

\begin{proof}
To determine $\tau(Z,\T)$ and $\varep(Z,\T)$, consider the generator $K$ (which is at $(i,j)
=(0,-1)$). The homology class of $K$ in $\HFp(Z)$ is a nontrivial element of both $\Im (\rho_* \circ
\iota_{-1*})$ and $\Im (\rho_* \circ v_{-1*})$ that is in the image of $U^N \HF^+(Z)$ for all $N \gg
0$, and moreover $v'_{-1*}([K]) \ne 0$. (See Definition \ref{def:taunu}.) From this observation, it is straightforward to verify that $\tau(Z,\T) = \nu(Z,\T) = \nu'(Z,\T) = -1$, and so $\varep(Z,\T) =0$.

For the computation of $\Upsilon_{Z,\T}$, we first note that $d(Z)=-2$, and the part of $\HFi(Z)$ in
grading $-2$ is generated by the cycle $\xi = K + UG + UF$. For any $t \in [0,2]$,
$\Upsilon_{Z,\T}(t)$ equals $-2$ times the minimal value of $s$ for which $\xi \in C_s^t(Z,\T)$.
When $t \in [0,1]$, we see that
\[
[\xi] \in \im (f_{s*}^t)  \quad \Longleftrightarrow \quad K \in C_s^t(Z,\T) \quad
\Longleftrightarrow \quad s \ge -\frac{t}{2},
\]
so $\Upsilon_{Z,\T}(t) = t$. The case when $t \in  [1,2]$ follows symmetrically.
\end{proof}

\begin{proof}[Proof of Proposition \ref{prop:epsilonzerotaunonzero}]
The homology sphere $Z$ does not bound a homology ball, since $d(Z)=-2$ as noted above. However, $Y
= Z \conn {-Z}$ does bound an integer homology ball, namely $(Z \smallsetminus B^3) \times I$. The
knot $K = \T \conn U \subset Z \conn -Z$ (where $U$ denotes the unknot in $-Z$) thus represents an element of $\CZhat$, and it has the same values
of $\tau$, $\varep$, and $\Upsilon$ as $(Z,\T)$, by the additivity properties of all three
invariants.
\end{proof}

The rest of this section is devoted to proving Proposition \ref{prop:cfkcore} using the filtered
mapping cone formula of \cite{HeddenLevine}, as described in Section~\ref{sec:surgery-cfk}. (Lipshitz, Ozsv\'ath, and Thurston have announced results providing a minus version of bordered Floer homology, with which we were able to give an alternate proof of Proposition \ref{prop:cfkcore}.)

\begin{figure}
\begin{tikzpicture}
	\begin{scope}[thin, black!20!white]
	\end{scope}
	\draw[step=1, black!30!white, very thin] (0, -3) grid (1, 4);

	\filldraw (0.65, 3.5) circle (2pt) node[] (a){};
	\filldraw (0.45, 2.5) circle (2pt) node[] (b){};
	\filldraw (0.65, 1.5) circle (2pt) node[] (c){};
	\filldraw (0.45, 1.5) circle (2pt) node[] (d){};
	\filldraw (0.7, 0.5) circle (2pt) node[] (e){};
	\filldraw (0.5, 0.5) circle (2pt) node[] (f){};
	\filldraw (0.3, 0.5) circle (2pt) node[] (g){};
	\filldraw (0.7, -0.5) circle (2pt) node[] (h){};
	\filldraw (0.5, -0.5) circle (2pt) node[] (i){};
	\filldraw (0.5, -1.5) circle (2pt) node[] (j){};	
	\filldraw (0.5, -2.5) circle (2pt) node[] (k){};	

	\draw [very thick, ->] (a) -- (c);
	\draw [very thick, ->] (b) -- (d);
	\draw [very thick, ->] (e) -- (h);
	\draw [very thick, ->] (f) -- (i);
	\draw [very thick, ->] (j) -- (k);

	\node [right] at (a) {$(2)$};
	\node [left] at (b) {$(1)$};
	\node [right] at (c) {$(1)$};
	\node [left] at (d) {$(0)$};
	\node [right] at (e) {$(-1)$};
	\node [above] at (f) {$(0)$};
	\node [left] at (g) {$(0)$};
	\node [right] at (h) {$(-2)$};
	\node [left] at (i) {$(-1)$};
	\node [right] at (j) {$(-3)$};
	\node [right] at (k) {$(-4)$};

\end{tikzpicture}
\caption{The $\Z$-filtered complex $\CFKh(S^3,T)$, where $T$ is the $(2,3)$ cable of the
left-handed trefoil.}
\label{fig:CFKhatcable}
\end{figure}

\begin{figure}
\subfigure[]{ \begin{tikzpicture}

	\begin{scope}[thin, black!20!white]
		\draw [<->] (-2, 0.5) -- (3, 0.5);
		\draw [<->] (0.5, -2) -- (0.5, 3);
	\end{scope}
	\draw[step=1, black!50!white, very thin] (-1.9, -1.9) grid (2.9, 2.9);
	
\foreach \x in {0}
{	
	\filldraw (\x+0.45, \x+0.65) circle (2pt) node[] (a){};
	\filldraw (\x-0.35, \x+0.65) circle (2pt) node[] (b){};
	\filldraw (\x-0.35, \x+2.35) circle (2pt) node[] (c){};
	\filldraw (\x+0.35, \x+2.35) circle (2pt) node[] (d){};
	\filldraw (\x+0.35, \x+1.35) circle (2pt) node[] (e){};
	\filldraw (\x+1.35, \x+1.35) circle (2pt) node[] (f){};
	\filldraw (\x+1.35, \x+0.35) circle (2pt) node[] (g){};
	\filldraw (\x+2.35, \x+0.35) circle (2pt) node[] (h){};
	\filldraw (\x+2.35, \x-0.35) circle (2pt) node[] (i){};
	\filldraw (\x+0.65, \x-0.35) circle (2pt) node[] (j){};	
	\filldraw (\x+0.65, \x+0.45) circle (2pt) node[] (k){};	
	\draw [very thick, ->] (c) -- (b);
	\draw [very thick, ->] (d) -- (c);
	\draw [very thick, ->] (d) -- (e);
	\draw [very thick, ->] (f) -- (e);
	\draw [very thick, ->] (f) -- (g);
	\draw [very thick, ->] (h) -- (g);
	\draw [very thick, ->] (h) -- (i);
	\draw [very thick, ->] (i) -- (j);
	\draw [very thick, ->] (k) -- (j);
	
	\node [above] at (a) {$a$};
	\node [left] at (b) {$b$};
	\node [above] at (c) {$c$};
	\node [above] at (d) {$d$};
	\node [above right] at (e) {$e$};
	\node [above] at (f) {$f$};
	\node [above right] at (g) {$g$};
	\node [above] at (h) {$h$};
	\node [right] at (i) {$i$};
	\node [below] at (j) {$j$};
	\node [right] at (k) {$k$};
}
\end{tikzpicture}
\label{subfig:pre}
} \qquad
\subfigure[]{
\begin{tikzpicture}

	\begin{scope}[thin, black!20!white]
		\draw [<->] (-2, 0.5) -- (3, 0.5);
		\draw [<->] (0.5, -2) -- (0.5, 3);
	\end{scope}
	\draw[step=1, black!50!white, very thin] (-1.9, -1.9) grid (2.9, 2.9);
	
\foreach \x in {0}
{	
	\filldraw (\x+0.45, \x+0.65) circle (2pt) node[] (a){};
	\filldraw (\x-0.35, \x+0.65) circle (2pt) node[] (b){};
	\filldraw (\x-0.35, \x+2.35) circle (2pt) node[] (c){};
	\filldraw (\x+0.35, \x+2.35) circle (2pt) node[] (d){};
	\filldraw (\x+0.35, \x+1.35) circle (2pt) node[] (e){};
	\filldraw (\x+1.35, \x+1.35) circle (2pt) node[] (f){};
	\filldraw (\x+1.35, \x+0.35) circle (2pt) node[] (g){};
	\filldraw (\x+2.35, \x+0.35) circle (2pt) node[] (h){};
	\filldraw (\x+2.35, \x-0.35) circle (2pt) node[] (i){};
	\filldraw (\x+0.65, \x-0.35) circle (2pt) node[] (j){};	
	\filldraw (\x+0.65, \x+0.45) circle (2pt) node[] (k){};	
	\draw [very thick, ->] (a) -- (b);
	\draw [very thick, ->] (c) -- (b);
	\draw [very thick, ->] (d) -- (c);
	\draw [very thick, ->] (d) -- (e);
	\draw [very thick, ->] (f) -- (e);
	\draw [very thick, ->] (f) -- (g);
	\draw [very thick, ->] (h) -- (g);
	\draw [very thick, ->] (h) -- (i);
	\draw [very thick, ->] (i) -- (j);
	\draw [very thick, ->] (k) -- (j);
	\draw [very thick, ->] (e) -- (b);
	\draw [very thick, ->] (f) -- (a);
	\draw [very thick, ->] (f) -- (k);
	\draw [very thick, ->] (g) -- (j);
	
	\node [above] at (a) {$a$};
	\node [left] at (b) {$b$};
	\node [above] at (c) {$c$};
	\node [above] at (d) {$d$};
	\node [above right] at (e) {$e$};
	\node [above] at (f) {$f$};
	\node [above right] at (g) {$g$};
	\node [above] at (h) {$h$};
	\node [right] at (i) {$i$};
	\node [below] at (j) {$j$};
	\node [right] at (k) {$k$};
}		
\end{tikzpicture}\label{subfig:post}
}
\caption{Left, a labeled $\mathbb{F}[U,U^{-1}]$-basis for $\CFKi(S^3,T)$ drawn in the $(i,j)$-plane.
Shown are the components of $\partial^\infty$ which are easily determined from $\CFKh(S^3,T)$ and
the symmetry between $C\{i=0\}$ and $C\{j=0\}$.  Note that $a$ and $k$ have the same Maslov and
Alexander grading.  Right, the full complex $\CFKi(S^3,T)$ expressed in terms of an
$\mathbb{F}[U,U^{-1}]$-basis.}
\label{fig:CFKicablegenerators}
\end{figure}

\begin{figure}
\begin{tikzpicture}
	\begin{scope}[thin, black!20!white]
		\draw [<->] (-3, 0.5) -- (4, 0.5);
		\draw [<->] (0.5, -3) -- (0.5, 4);
	\end{scope}
	\draw[step=1, black!50!white, very thin] (-2.9, -2.9) grid (3.9, 3.9);

\foreach \x in {-2,...,1}
{	
	\filldraw (\x+0.45, \x+0.65) circle (2pt) node[] (a){};
	\filldraw (\x-0.35, \x+0.65) circle (2pt) node[] (b){};
	\filldraw (\x-0.35, \x+2.35) circle (2pt) node[] (c){};
	\filldraw (\x+0.35, \x+2.35) circle (2pt) node[] (d){};
	\filldraw (\x+0.35, \x+1.35) circle (2pt) node[] (e){};
	\filldraw (\x+1.35, \x+1.35) circle (2pt) node[] (f){};
	\filldraw (\x+1.35, \x+0.35) circle (2pt) node[] (g){};
	\filldraw (\x+2.35, \x+0.35) circle (2pt) node[] (h){};
	\filldraw (\x+2.35, \x-0.35) circle (2pt) node[] (i){};
	\filldraw (\x+0.65, \x-0.35) circle (2pt) node[] (j){};	
	\filldraw (\x+0.65, \x+0.45) circle (2pt) node[] (k){};	
	\draw [very thick, ->] (a) -- (b);
	\draw [very thick, ->] (c) -- (b);
	\draw [very thick, ->] (d) -- (c);
	\draw [very thick, ->] (d) -- (e);
	\draw [very thick, ->] (f) -- (e);
	\draw [very thick, ->] (f) -- (g);
	\draw [very thick, ->] (h) -- (g);
	\draw [very thick, ->] (h) -- (i);
	\draw [very thick, ->] (i) -- (j);
	\draw [very thick, ->] (k) -- (j);
	\draw [very thick, ->] (e) -- (b);
	\draw [very thick, ->] (f) -- (a);
	\draw [very thick, ->] (f) -- (k);
	\draw [very thick, ->] (g) -- (j);
}	

\end{tikzpicture}
\caption{The complex $\CFKi(K)$ expressed in terms of an $\mathbb{F}$-basis.}
\label{fig:CFKicable}
\end{figure}

The first step is understanding the filtered complex $\CFKi(S^3,T)$. From \cite[Section
5]{Petkovacables} (see also \cite[Table 1.0.6]{Heddenthesis}), $\CFKh(S^3,T)$ has rank 11, with
gradings and differentials as depicted in Figure \ref{fig:CFKhatcable}.  Therefore, we obtain a
basis over $\mathbb{F}[U,U^{-1}]$ for $\CFKi(S^3,T)$ with generators labeled and drawn in the
$(i,j)$-plane as in Figure~\ref{subfig:pre}.  Further, we have computed the subquotient complex
$C\{i=0\}$ inside of $\CFKi(S^3,T)$.  The symmetry between $C\{i=0\}$ and $C\{j=0\}$ determines $C\{j =
0\}$, except for the component of $\partial^\infty$ from $C\{(0,0)\}$ to $C\{(-1,0)\}$, as the
generators $a$ and $k$ have the same Alexander and Maslov gradings.  By $U$-equivariance, a large
portion of the differential $\partial^\infty$ on $\CFKi(S^3,T)$ is consequently determined.  This
information is summarized in Figure~\ref{subfig:pre}. Note that the $U$-exponent of a element does not determine the $\mathcal{I}$-filtration since our basis over $\F[U, U^{-1}]$ is not concentrated in $\mathcal{I}$-filtration zero; the same statement applies later to our calculation of $\CFKi(Z,\T)$.

We must determine the rest of the differential on $\CFKi(S^3,T)$.  We start by completing the
differential on the subquotient complex $C\{j = 0\}$, where it remains to determine the
differentials of $a$ and $k$.  The symmetry between $C\{i = 0\}$ and $C\{j=0\}$ shows that the
differential of either $a$ or $k$ (or both) must be $b$.  Since $\varep(T) = -1$ \cite[Theorem
2]{Homcables}, it must be that, up to a change of basis, the differential sends $a$ to $b$ and $k$
to 0 by \cite[Lemma 3.2]{Homcables}.  Therefore, the differentials which preserve one of the two
$\mathbb{Z}$-filtrations have been computed, and it remains to determine differentials which
strictly lower both filtration levels, i.e., ``diagonal arrows''.   By grading considerations (i.e., that the differential lowers grading by $1$) and
$(\d^\infty)^2=0$, the remaining components of $\d^\infty$ are completely determined; the end result
is shown in Figure~\ref{fig:CFKicablegenerators} (expressed in an $\mathbb{F}[U,U^{-1}]$-basis) and
Figure~\ref{fig:CFKicable} (expressed in an $\mathbb{F}$-basis).

Next, we determine the self-chain homotopy equivalence $\Psi^\infty \co C \to C$ identifying $C\{i \leq0\}$
and $C\{j\leq0\}$.  Since $H_*(C\{i\leq0\}) \cong H_*(C\{j\leq0\}) \cong \F[U]$, $\Psi^\infty$ is unique up
to chain homotopy (cf. \cite[Lemma 2.14]{HeddenLevine}). The $\F[U]$-equivariant map on $\CFKi$ which fixes $f$ and interchanges
\[ a \leftrightarrow k  \qquad b \leftrightarrow j  \qquad c \leftrightarrow i  \qquad d \leftrightarrow h  \qquad e \leftrightarrow g \]
induces such a chain homotopy equivalence, so we may assume that $\Psi^\infty$ is given by this map. (The constructions in the filtered mapping cone formula depend only on the chain homotopy type of $\Psi^\infty$.)

We now apply Theorem \ref{thm:surgery} to compute $\CFKi(Z,\T)$. Since $g(T)=3$, the complex
$\X^\infty$ is a mapping cone
\[
\bigoplus_{s=-2}^3 A_s^\infty \to \bigoplus_{s=-1}^3 B_s^\infty,
\]
with $\II$ and $\JJ$ filtrations as defined in Section \ref{sec:surgery-cfk}. As discussed above, in
order to pass from $\X^\infty$ to a reduced model for $\CFKi(Z,\T)$, we begin by considering
$\CFKh(Z,\T)$.

We will work out $\CFKh(Z,\T,1)$ explicitly; the remaining computations are similar and are left to the
reader. By \eqref{eq:hat-mapping-cone}, $\CFKh(Z,\T,1)$ is given by the mapping cone
\begin{equation} 
 A_1\{i \le 0, j=1\} \oplus A_2\{i = 0, j \le 1\} \xrightarrow {(h_1, \ v_2)} B_2\{i=0\}.
\end{equation}
We depict $A_1$, $A_2$, and $B_2$ in Figures \ref{subfig:A1}, \ref{subfig:A2}, and \ref{subfig:B2},
respectively. Generators of $A_i$ are labeled with an $i$-subscript, and generators of $B_i$ are labeled with a prime and an $i$-subscript. The dark grey shaded regions in Figure
\ref{fig:CFK_1core} thus determine $\CFKh(Z,\T,1)$. The generators and differentials of
$\CFKh(Z,\T,1)$ are as follows:
\begin{align*}
\d (U^{-1}b_1) &= j'_2 &	    \d (U^{-1}b_2) &= U^{-1}b'_2 		&\d (U^{-1}c'_2) &= U^{-1}b'_2 \\
\d (e_1) &= Ug'_2 &           \d (e_2) &= 	e'_2				&\d (d'_2) &= e'_2 \\
\d (Ud_1) &= U^2h'_2+Uc_1 &   \d (a_2) &=	a'_2				&\d (U^{-1}b'_2) &= 0 \\
\d (Uc_1) &= U^2i'_2 &        \d (k_2) &=	k'_2+j_2			&\d (e'_2) &= 0 \\
&&                          \d (Uf_2) &= Uf'_2 +Ug_2			&\d (a_2') &= 0 \\
&&                          \d (j_2) &= j'_2					&\d (k'_2) &= j'_2 \\
&&                          \d (Ug_2) &= Ug'_2				&\d (Uf'_2) &= Ug'_2 \\
&&                          \d (U^2h_2) &= U^2h'_2 +U^2i_2	&\d (j'_2) &= 0 \\
&&                          \d (U^2i_2) &= U^2i'_2	 		&\d (Ug'_2) &= 0 \\
&&			                &                       		&\d (U^2h'_2) &= U^2i'_2 \\
&&			                &                       		&\d (U^2i'_2) &= 0.
\end{align*}

It follows that $\HFKh(Z,\T,1)$ is generated by
\begin{equation} \label{eq:HFK-1-gens}
\{U^{-1}b_2+U^{-1}c'_2, \quad e_2+d'_2, \quad U^{-1}b_1+k'_2, \quad e_1+Uf'_2\}.
\end{equation}
We may likewise compute $\HFKh(Z,\T)$ in the other Alexander gradings using the same techniques. The
results are summarized in the first three columns of Table \ref{tab:CFKinftycore}.

\begin{figure}[htb!]
\subfigure[]{
\begin{tikzpicture}[scale=0.85]

\filldraw[black!5!white] (0, -3) rectangle (1, 1);
\filldraw[black!25!white] (-3, 1) rectangle (1, 2);

	\begin{scope}[thin, black!20!white]
		\draw [<->] (-3, 0.5) -- (4, 0.5);
		\draw [<->] (0.5, -3) -- (0.5, 4);
	\end{scope}
	\draw[step=1, black!50!white, very thin] (-2.9, -2.9) grid (3.9, 3.9);

\foreach \x in {-2,...,1}
{	
	\filldraw (\x+0.45, \x+0.65) circle (2pt) node[] (a){};
	\filldraw (\x-0.35, \x+0.65) circle (2pt) node[] (b){};
	\filldraw (\x-0.35, \x+2.35) circle (2pt) node[] (c){};
	\filldraw (\x+0.35, \x+2.35) circle (2pt) node[] (d){};
	\filldraw (\x+0.35, \x+1.35) circle (2pt) node[] (e){};
	\filldraw (\x+1.35, \x+1.35) circle (2pt) node[] (f){};
	\filldraw (\x+1.35, \x+0.35) circle (2pt) node[] (g){};
	\filldraw (\x+2.35, \x+0.35) circle (2pt) node[] (h){};
	\filldraw (\x+2.35, \x-0.35) circle (2pt) node[] (i){};
	\filldraw (\x+0.65, \x-0.35) circle (2pt) node[] (j){};	
	\filldraw (\x+0.65, \x+0.45) circle (2pt) node[] (k){};	
	\draw [thick, ->] (a) -- (b);
	\draw [thick, ->] (c) -- (b);
	\draw [thick, ->] (d) -- (c);
	\draw [thick, ->] (d) -- (e);
	\draw [thick, ->] (f) -- (e);
	\draw [thick, ->] (f) -- (g);
	\draw [thick, ->] (h) -- (g);
	\draw [thick, ->] (h) -- (i);
	\draw [thick, ->] (i) -- (j);
	\draw [thick, ->] (k) -- (j);
	\draw [thick, ->] (e) -- (b);
	\draw [thick, ->] (f) -- (a);
	\draw [thick, ->] (f) -- (k);
	\draw [thick, ->] (g) -- (j);
}		
\node[left] at (-1.35, 1.35) {$Uc_1$};
\node[above,xshift=-10] at (-0.35, 1.35) {$Ud_1$};
\node[left] at (0.35, 1.35) {$e_1$};
\node[left] at (0.78, 1.85) {\scriptsize$U^{-1}b_1$};
\end{tikzpicture}
\label{subfig:A1}
}
\subfigure[]{
\begin{tikzpicture}[scale=0.85]

\filldraw[black!25!white] (0, -3) rectangle (1, 2);
\filldraw[black!05!white] (-3, 2) rectangle (1, 3);

	\begin{scope}[thin, black!20!white]
		\draw [<->] (-3, 0.5) -- (4, 0.5);
		\draw [<->] (0.5, -3) -- (0.5, 4);
	\end{scope}
	\draw[step=1, black!50!white, very thin] (-2.9, -2.9) grid (3.9, 3.9);

\foreach \x in {-2,...,1}
{	
	\filldraw (\x+0.45, \x+0.65) circle (2pt) node[] (a){};
	\filldraw (\x-0.35, \x+0.65) circle (2pt) node[] (b){};
	\filldraw (\x-0.35, \x+2.35) circle (2pt) node[] (c){};
	\filldraw (\x+0.35, \x+2.35) circle (2pt) node[] (d){};
	\filldraw (\x+0.35, \x+1.35) circle (2pt) node[] (e){};
	\filldraw (\x+1.35, \x+1.35) circle (2pt) node[] (f){};
	\filldraw (\x+1.35, \x+0.35) circle (2pt) node[] (g){};
	\filldraw (\x+2.35, \x+0.35) circle (2pt) node[] (h){};
	\filldraw (\x+2.35, \x-0.35) circle (2pt) node[] (i){};
	\filldraw (\x+0.65, \x-0.35) circle (2pt) node[] (j){};	
	\filldraw (\x+0.65, \x+0.45) circle (2pt) node[] (k){};	
	\draw [thick, ->] (a) -- (b);
	\draw [thick, ->] (c) -- (b);
	\draw [thick, ->] (d) -- (c);
	\draw [thick, ->] (d) -- (e);
	\draw [thick, ->] (f) -- (e);
	\draw [thick, ->] (f) -- (g);
	\draw [thick, ->] (h) -- (g);
	\draw [thick, ->] (h) -- (i);
	\draw [thick, ->] (i) -- (j);
	\draw [thick, ->] (k) -- (j);
	\draw [thick, ->] (e) -- (b);
	\draw [thick, ->] (f) -- (a);
	\draw [thick, ->] (f) -- (k);
	\draw [thick, ->] (g) -- (j);
}		
\node[] at (1.2, 2) {\tiny$U^{-1}b_2$};
\node[] at (0.1, 1.4) {\small$e_2$};
\node[] at (0.4, 0.9) {\scriptsize$a_2$};
\node[] at (0.9, 0.4) {\scriptsize$k_2$};
\node[] at (0, 0.5) {\tiny$Uf_2$};
\node[] at (0.5, -0.14) {\scriptsize$j_2$};
\node[below] at (0.4, -0.65) {\small$U g_2$};
\node[right] at (0.35, -1.75) {$U^2 h_2$};
\node[right] at (0.35, -2.45) {$U^2 i_2$};
\end{tikzpicture}
\label{subfig:A2}
}
\subfigure[]{
\begin{tikzpicture}[scale=0.85]

	\filldraw[black!25!white] (0, -3) rectangle (1, 4);

	\draw[step=1, black!50!white, very thin] (-2.9, -2.9) grid (3.9, 3.9);

	\begin{scope}[thin, black!20!white]
		\draw [<->] (-3, 0.5) -- (4, 0.5);
		\draw [<->] (0.5, -3) -- (0.5, 4);
	\end{scope}

\foreach \x in {-2,...,1}
{	
	\filldraw (\x+0.45, \x+0.65) circle (2pt) node[] (a){};
	\filldraw (\x-0.35, \x+0.65) circle (2pt) node[] (b){};
	\filldraw (\x-0.35, \x+2.35) circle (2pt) node[] (c){};
	\filldraw (\x+0.35, \x+2.35) circle (2pt) node[] (d){};
	\filldraw (\x+0.35, \x+1.35) circle (2pt) node[] (e){};
	\filldraw (\x+1.35, \x+1.35) circle (2pt) node[] (f){};
	\filldraw (\x+1.35, \x+0.35) circle (2pt) node[] (g){};
	\filldraw (\x+2.35, \x+0.35) circle (2pt) node[] (h){};
	\filldraw (\x+2.35, \x-0.35) circle (2pt) node[] (i){};
	\filldraw (\x+0.65, \x-0.35) circle (2pt) node[] (j){};	
	\filldraw (\x+0.65, \x+0.45) circle (2pt) node[] (k){};	
	\draw [thick, ->] (a) -- (b);
	\draw [thick, ->] (c) -- (b);
	\draw [thick, ->] (d) -- (c);
	\draw [thick, ->] (d) -- (e);
	\draw [thick, ->] (f) -- (e);
	\draw [thick, ->] (f) -- (g);
	\draw [thick, ->] (h) -- (g);
	\draw [thick, ->] (h) -- (i);
	\draw [thick, ->] (i) -- (j);
	\draw [thick, ->] (k) -- (j);
	\draw [thick, ->] (e) -- (b);
	\draw [thick, ->] (f) -- (a);
	\draw [thick, ->] (f) -- (k);
	\draw [thick, ->] (g) -- (j);
}		
\node[] at (0.7, 3.7) {$U^{-1}c'_2$};
\node[] at (0.3, 2.7) {$d'_2$};
\node[] at (1.2, 2) {\tiny$U^{-1}b'_2$};
\node[] at (0.1, 1.4) {\small$e'_2$};
\node[] at (0.4, 0.9) {\scriptsize$a'_2$};
\node[] at (0.9, 0.4) {\scriptsize$k'_2$};
\node[] at (0, 0.5) {\tiny$Uf'_2$};
\node[] at (0.5, -0.14) {\scriptsize$j'_2$};
\node[below] at (0.4, -0.65) {\small$U g'_2$};
\node[right] at (0.35, -1.75) {$U^2 h'_2$};
\node[right] at (0.35, -2.45) {$U^2 i'_2$};
\end{tikzpicture}
\label{subfig:B2}
}
\caption{Top left, $A_1$. Top right, $A_2$. Bottom, $B_2$. The dark grey regions comprise
$\CFKh(Z,\T,1)$.}
\label{fig:CFK_1core}
\end{figure}

\begin{table}[htb!]
\centering
\renewcommand{\arraystretch}{1.2}
\begin{tabular}{rcccl}
\hline
& Generator & Bigrading & $\d^\infty$ & \\ \hline
$A=$ & $U^{-1}c_3$ & $(3,8)$ & $U^{-1}b_3+U^{-1}c'_3$ & $=B$\\
$B=$ & $U^{-1}b_3+U^{-1}c'_3$ & $(2,7)$ & $0$ & \\
$C=$ & $U^{-1}b_2+U^{-1}c'_2$ & $(1,3)$ & $(Uj'_3)$  & $\equiv 0$ \\
$D=$ & $e_2+d'_2 $& $(1,2)$ & $b_2+c'_2 \; + ( U^2g'_3)$ & $\equiv UC$ \\
$E=$ & $U^{-1}b_1+k'_2$ & $(1,1)$ & $(U^{-1}b'_1)$ & $\equiv 0$ \\
$F=$ & $e_1+ Uf'_2$ & $(1,0)$ & \makecell{$b_1+Uk'_2 $ \\ $ + \; ( Ue'_2+Ua'_2+e'_1)$} & $\equiv UE$  \\
$G=$ & $k_0+a_1$ & $(0,0)$ & $j_0+k'_0+b_1+Uk'_2$ & $= J + UE$ \\
$H=$ & $j_{-1}+U^{-1}c'_0$& $(-1,1)$ & $(j'_{-1})$ & $\equiv 0$ \\
$I=$ & $Ug_{-1}+d'_0$ & $(-1,0)$ & $Uj_{-1}+c'_0 \; + [Ug'_{-1}]$ & $\equiv UH + [U^2L]$ \\
$J=$ & $j_0+k'_0$ & $(-1,-1)$ & $(b'_1)$ & $\equiv 0$ \\
$K=$ & $Ug_0+Uf'_0$ & $(-1,-2)$ & \makecell{$Uj_0+Uk'_0$ \\ $+ \; (Ue'_0+Ue'_1) + [Ua'_0]$} & $\equiv UJ + [U^2H]$ \\
$L=$ & $Uj_{-2}+U^{-1}c'_{-1}$ & $(-2,3)$ & $0$ & \\
$M=$ & $U^2 i_{-2}$ & $(-3,2)$ & $U^2j_{-2}+c'_{-1}$ & $=UL$\\ \hline
\end{tabular}
\caption{Summary of the generators of $\X^\infty$ which survive in the reduced model for $\CFKi(Z,
\T)$. The bigrading is (Alexander, Maslov). In the $\partial^\infty$ column, the terms in parentheses become trivial in the reduced model, while the terms in square brackets become homologous to the bracketed terms on the right, leading to the congruences shown. The differentials are shown graphically in Figure \ref{fig:CFKinftycore2}, below.}
\label{tab:CFKinftycore}
\end{table}

Next, we compute a reduced model for $\CFKi(Z, \T)$ by the method described at the end of Section
\ref{sec:surgery-cfk}.  Recall that we consider the generators of $\HFKh(Z, \T)$ as generators
(over
$\F[U,U^{-1}]$) of the reduced model for $\CFKi(Z, \T)$ and compute the induced differential, each
term of which lowers at least one of the $\II$ and $\JJ$ filtrations. More precisely, we consider the
basis for the (unreduced) complex $\X^\infty$ given by the filtered mapping cone and choose
representatives for the generators of $\HFKh(Z, \T)$. The differential $\d^\infty$ on $\CFKi(Z,
\T)$
decomposes as a sum $\d^\infty=\d + \d'$, where $\d$ preserves the bifiltration, and $\d'$ lowers at
least one of the filtrations. Choose a basis $\{y_i\}$ for $\Im \d$, and for each $y_i$, choose an
$x_i$ such that $\d x_i = y_i$. We quotient $\CFKi(Z, \T)$ by the $\F[U,U^{-1}]$--submodule spanned
by $\{ x_i, \d^\infty x_i\}$. This produces the reduced model for $\CFKi(Z,\T)$, and we can write
the induced differential on the reduced model for $\CFKi(Z, \T)$ purely in terms of generators of
$\HFKh(Z,\T)$.

\begin{figure} [htb!]
\begin{tikzpicture}

	\begin{scope}[thin, black!20!white]
		\draw [<->] (-3, 0.5) -- (4, 0.5);
		\draw [<->] (0.5, -4) -- (0.5, 4);
	\end{scope}
	\draw[step=1, black!50!white, very thin] (-2.9, -3.9) grid (3.9, 3.9);
	
	\filldraw (-1.5, 1.5) circle (2pt) node[] (A){};
	\filldraw (-1.5, 0.5) circle (2pt) node[] (B){};
	\filldraw (-0.65, 0.35) circle (2pt) node[] (D){};	
	\filldraw (-1.5, -0.5) circle (2pt) node[] (C){};
	\filldraw (0.55, 1.5) circle (2pt) node[] (F){};	
	\filldraw (-0.45, 0.55) circle (2pt) node[] (E){};	
	\filldraw (0.55, 0.55) circle (2pt) node[] (G){};
	\filldraw (0.55,-0.45) circle (2pt) node[] (J){};
	\filldraw (1.5, 0.55) circle (2pt) node[] (K){};
	\filldraw (0.35, -0.65) circle (2pt) node[] (I){};
	\filldraw (-0.5, -1.5) circle (2pt) node[] (H){};
	\filldraw (-0.5, -3.5) circle (2pt) node[] (M){};
	\filldraw (-1.5, -3.5) circle (2pt) node[] (L){};

	\draw [very thick, ->] (A) -- (B);
	\draw [very thick, ->] (D) -- (C);
	\draw [very thick, ->] (F) -- (E);
	\draw [very thick, ->] (G) -- (E);
	\draw [very thick, ->] (G) -- (J);
	\draw [very thick, ->] (I) -- (H);
	\draw [very thick, ->] (K) -- (J);
	\draw [very thick, ->] (M) -- (L);
	\path [very thick, ->, bend left=35] (K) edge node[]{} (H);
	\path [very thick, ->, bend left=15] (I) edge node[]{} (L);

	\node [above] at (A) {$U^2A$};
	\node [left] at (B) {$U^2B$};
	\node [left] at (C) {$U^2C$};
	\node [below,xshift=5] at (D) {$UD$};
	\node [above,xshift=-5] at (E) {$UE$};
	\node [above] at (F) {$F$};
	\node [above] at (G) {$G$};
	\node [left] at (H) {$UH$};
	\node [left] at (I) {$I$};
	\node [right,xshift=1] at (J) {$J$};
	\node [right] at (K) {$U^{-1}K$};
	\node [below] at (L) {$U^2L$};
	\node [right] at (M) {$UM$};

\end{tikzpicture}
\caption{The reduced model for $\CFKi(Z,\T)$, drawn in the $(i,j)$-plane, before the change of basis resulting in Figure \ref{fig:CFKinftycore}. }
\label{fig:CFKinftycore2}
\end{figure}

We will carry out the strategy above explicitly for the generators of $\HFKh(Z,\T)$ with Alexander
grading 1 (listed in \eqref{eq:HFK-1-gens}) with the help of Figure \ref{fig:CFK_1core}. We have
\begin{align*}
	\d^\infty (U^{-1}b_2+U^{-1}c'_2) &= Uj'_3 \\
	\d^\infty (e_2+d'_2) &= b_2+c'_2+U^2g'_3  \\
	\d^\infty(U^{-1}b_1+k'_2) &= U^{-1}b'_1 \\
	\d^\infty (e_1 + U f'_2) &= b_1+Uk'_2 + Ue'_2+Ua'_2+e'_1.
\end{align*}
In order to compute the induced differential on the reduced model for $\CFKi(Z, \T)$, we note that
\begin{equation}\label{eqn:imd}
U^{-1}b'_1, \ e'_1, \ a'_2, \ e'_2, \ j'_3, \ Ug'_3, \ Uk'_3+Uj_3
\end{equation}
are all in the image of $\d$. We choose preimages (under $\d$) for each of the elements in \eqref{eqn:imd}, and compute $\d^\infty$ of these preimages:

\begin{align*}
\d (U^{-1}c'_1) &= U^{-1}b'_1 			&\d^\infty (U^{-1}c'_1) &= U^{-1}b'_1 \\
\d (d'_1) &= e'_1		 				&\d^\infty (d'_1) &= e'_1+c'_1 \\
\d (a_2) &= a'_2						&\d^\infty (a_2) &= a'_2+U^2k'_3+b_2 \\
\d (e_2) &= e'_2						&\d^\infty (e_2) &= e'_2+U^2g'_3+b_2\\
\d (j_3) &= j'_3						& \d^\infty (j_3) &= j'_3 \\
\d (Ug_3) &= Ug'_3					& \d^\infty (Ug_3) &= Ug'_3+Uj_3 \\
\d (Uk_3) &= Uk'_3 +Uj_3				& \d^\infty (Uk_3) &= Uk'_3+Uj_3. \\
\end{align*}

Thus, after quotienting by the $\F[U,U^{-1}]$-submodule $S$ generated by
\begin{multline*}
\{ U^{-1}c'_1, \ \d^\infty U^{-1}c'_1, \ d'_1, \ \d^\infty d'_1, \ a_2, \ \d^\infty a_2, \ e_2, \ \d^\infty e_2, \\ j_3, \ \d^\infty j_3, \ Ug_3, \ \d^\infty Ug_3, \ Uk_3, \ \d^\infty Uk_3 \},
\end{multline*}
it readily follows that $Uj'_3$ and $U^{-1}b'_1$ are trivial. Similarly, we have that $U^2g'_3$ is trivial, as $Ug'_3+Uj_3$ and $j_3$ are both in $S$. We also have that $Ue'_2+Ua'_2+e'_1$ is trivial, as the elements
\[
a'_2+U^2k'_3+b_2, \ e'_2+U^2g'_3+b_2, \ Uk'_3+Uj_3, \ Ug'_3+Uj_3, \ e'_1+c'_1, \ U^{-1}c'_1
\]
are all in $S$.

As the elements in \eqref{eqn:imd} are linearly independent, they can be completed to a basis for
$\Im \d$, and we may continue this procedure to compute the induced differential on a reduced model
for $\CFKi(Z, \T)$.

The calculations in the remaining Alexander gradings follow similarly and are summarized in Table \ref{tab:CFKinftycore}. We need to verify that the remaining terms in parentheses in the $\d^\infty$ column of Table \ref{tab:CFKinftycore} are indeed trivial in the reduced model and to express the terms in the square brackets in terms of generators of homology. Note that
\begin{equation}\label{eq:S'}
   a'_{-1}, \ e'_{-1}, \ Ug'_{-1}, \ j'_{-1}, \ a'_0, \ U^{-1}b'_0, \ e'_0
\end{equation}
are all in the image of $\d$. As above, we choose preimages (under $\d$) for each of the elements in \eqref{eq:S'}, and compute their images under $\d^\infty$:
\begin{align*}
\d (U^2 k_{-2}) &= a'_{-1}			&\d^\infty (U^2 k_{-2}) &= a'_{-1} + U^2j_{-2} \\
\d (d'_{-1}) &= e'_{-1}			&\d^\infty (d'_{-1}) &= e'_{-1}+c'_{-1} \\
\d (Uf'_{-1}) &= Ug'_{-1}			&\d^\infty (Uf'_{-1}) &= Ug'_{-1}	+Ua'_{-1}+Ue'_{-1}+Uk'_{-1} \\
\d (k'_{-1}) &= j'_{-1}			&\d^\infty (k'_{-1}) &= j'_{-1} \\
\d (Uk_{-1}) &= a'_0				&\d^\infty (Uk_{-1}) &= a'_0 + Uj_{-1} + Uk'_{-1} \\
\d (U^{-1}c'_0) &= U^{-1}b'_0				&\d^\infty (U^{-1}c'_0) &= U^{-1}b'_0 \\
\d (d'_0) &= e'_0,				&\d^\infty (d'_0) &= e'_0 + c'_0.
\end{align*}
Let $S'$ be the $\F[U,U^{-1}]$-submodule generated by $S$ together with
\begin{multline*}
\{ k_{-2}, \ \d^\infty k_{-2}, \ d'_{-1}, \ \d^\infty d'_{-1}, \ f'_{-1}, \ \d^\infty f'_{-1}, \ k'_{-1}, \ \d^\infty k'_{-1}, \\
 k_{-1}, \ \d^\infty k_{-1}, \ c'_0, \ \d^\infty c'_0, \ d'_0, \ \d^\infty d'_0 \}.
\end{multline*}
We assume that we quotient by $S'$ while passing to the reduced model. Note that $j'_{-1}$, $b'_1$, $U e'_0$, and $Ue'_1$ are all in $S'$; thus, the remaining terms in parentheses in Table \ref{tab:CFKinftycore} become trivial in the reduced model.


To see that $Ug'_{-1}$ is equivalent to $U^2L$ in the reduced model, we note that we have the following equivalences:
\[ Ug'_{-1} \equiv Ua'_{-1}+Ue'_{-1}+Uk'_{-1} \equiv Ua'_{-1}+Ue'_{-1} \equiv U^3j'_{-2}+Uc'_1 \equiv U^2L, \]
since the elements
\[
Ug'_{-1}+Ua'_{-1}+Ue'_{-1}+Uk'_{-1}, \ Uk'_{-1}, \  Ua'_{-1}+U^3j'_{-2}, \ Ue'_{-1}+Uc'_1
\]
are all in $S'$. Likewise, we have that
\[ Ua'_0 \equiv U^2j_{-1} +U^2k'_{-1}\equiv U^2j_{-1} + U c'_0 \equiv UH \]
since $Ua'_0 + U^2j_{-1}+U^2k'_{-1}$, $U^2k'_{-1}$, and $Uc'_0$ are in $S'$.

The differential on this reduced model for $\CFKi(Z, \T)$ is summarized in the rightmost column in
Table \ref{tab:CFKinftycore} and graphically in Figure \ref{fig:CFKinftycore2}. To obtain Figure \ref{fig:CFKinftycore} from Figure \ref{fig:CFKinftycore2}, we perform a filtered change of basis, replacing $I$ with $I+UM$ and $K$ with $K +UI+U^2M$. (By slight abuse of notation, we use the same symbols.) This completes the proof of Proposition \ref{prop:cfkcore}.

%

\appendix
\section{Smoothings of PL surfaces in 4-manifolds} \label{app:PL-smooth}

The purpose of this section is to prove the following folklore theorem:

\begin{theorem} \label{thm:PL-smooth}
Let $\Sigma$ be a compact, piecewise-linear surface in a compact, smooth $4$-manifold $X$. If $\Sigma$ has nonempty boundary, assume that $\partial \Sigma = \partial X \cap \Sigma$. Then there is a piecewise-differentiable isotopy taking $\Sigma$ to a surface that is smooth except possibly at finitely many singular points, each of which is smoothly modeled on the cone of a smooth knot in $S^3$.
\end{theorem}

By way of background, earlier papers that dealt with non--locally flat surfaces in $4$-manifolds (e.g. \cite{ZeemanDunce, MatsumotoSpine}) stayed entirely within the PL category. More recent results on the subject (e.g. \cite{AkbulutZeeman, Levinenonsurj, LevineLidmanSpineless}) have been stated in terms of PL surfaces in smooth $4$-manifolds, since that is the setting in which gauge theory and Floer homology are defined, and have implicitly made use of Theorem \ref{thm:PL-smooth}. It is well-known that any locally flat PL surface in a smooth $4$-manifold is isotopic to a smoothly embedded surface; this is a special case of a more general theorem of Wall \cite{WallCodimension2} that holds for any codimension-$2$ PL submanifold of a smooth manifold. However, we have not managed to find a precise statement of Theorem \ref{thm:PL-smooth} (taking into account the singular points) in the literature. At the referee's suggestion, we therefore provide a proof.

Throughout, we will assume that $\Sigma$ is a PL surface in a smooth $4$-manifold $X$, as in the statement of the theorem. Formally, this means that $\Sigma$ is a $2$-dimensional subcomplex in some smooth triangulation of $X$. The surface $\Sigma$ need not be locally flat; it may have singular points $p_1, \dots, p_n$ that are combinatorially modeled on the cones of knots in $S^3$. That is, for each $p_i$, there is a PL homeomorphism from $(D^4, \Cone(K_i))$ to $(V_i,V_i \cap \Sigma)$, where $V_i$ is a closed neighborhood of $p_i$ and $K_i$ is a PL knot in $S^3$. The first claim is that after performing a small isotopy, we can arrange that each $p_i$ has a neighborhood that is in fact smoothly modeled on the cone of $K_i$. The precise statement is as follows:

\begin{lemma} \label{lemma:vertexcone}
Let $\Sigma$ be a PL surface in a smooth $4$-manifold $X$, with $\partial \Sigma \subset \partial X$. After subdivision and piecewise-differentiable isotopy of the triangulation of $X$, we may assume that the following holds: For each vertex $p$ at which $\Sigma$ is not locally flat, there is a smooth coordinate ball $(U,\phi)$ around $p$, disjoint from all other vertices, such that $\phi(U \cap \Sigma)$ is the cone of a PL knot $K \subset S^3$.
\end{lemma}

\begin{proof}
Choose a smooth triangulation of $X$ such that $\Sigma$ is a subcomplex of the $2$-skeleton of $X$. By definition, the triangulation consists of a homeomorphism $h$ from a finite simplicial complex $T$ to $X$ whose restriction to each simplex is a smooth immersion.

Let $p_1, \dots, p_n \in T$ be the vertices where $\Sigma$ fails to be locally flat. (We will view $\Sigma$ as a subset of $T$ rather than $X$.) After sufficiently subdividing $T$, we may assume that the following holds: if $B_i = \St(T,p_i)$, $C_i = \St(T,B_i)$, and $D_i = \St(T,C_i)$, then $D_i$ is a combinatorial $4$-ball, and we may find disjoint charts $(U_i, \phi_i)$ (in the smooth structure of $X$) such that $D_i \subset U_i$. (Here $\St$ denotes the star of a subcomplex.) Observe that $D_i \cap \Sigma$ is (a subdivision of) the cone on some PL knot $K_i$ appearing in the $1$-skeleton of this triangulation. Furthermore, $\Sigma$ is still locally flat at any new vertices introduced in the course of the subdivision; thus, $p_1, \dots, p_n$ are still the only vertices where $\Sigma$ is not locally flat.

Let $\tilde h_i = \phi_i \circ h|_{D_i} \co D_i \to \R^4$; this is a $C^\infty$ embedding of the $4$-dimensional simplicial complex $D_i$ in $\R^4$, as in \cite[Definition 8.3]{MunkresDifferential}.

We now apply the idea of approximating by the secant map, as in \cite[Chapter 9]{MunkresDifferential}. Namely, by \cite[Theorem 9.7]{MunkresDifferential}, for any $\delta>0$, we may find a subdivision $D_i^{(\delta)}$ of $D_i$, and a smooth map $\tilde h_i^{(\delta)} \co D_i^{(\delta)} \to \R^4$, with the following properties:
\begin{enumerate}
\item Outside of $C_i$, we have $D_i^{(\delta)} = D_i$ and $\tilde h_i^{(\delta)} = \tilde h_i$. (That is, simplices of $D_i$ not contained in $C_i$ are not subdivided further in $D_i^{(\delta)}$, and $\tilde h_i^{(\delta)} = \tilde h_i$ on these simplices.)

\item For each simplex $\sigma$ of $B_i^{(\delta)}$, $f_i^{(\delta)}|_\sigma$ equals the linear map induced by the inclusion of the vertices of $\sigma$.

\item $\tilde h_i^{(\delta)}$ is a $\delta$-approximation to $\tilde h_i$.

\item $\im(\tilde h_i^{(\delta)}) = \im(\tilde h_i)$.
\end{enumerate}
(For statement 4, we use the fact that $\tilde h_i = \tilde h_i^{(\delta)}$ on $\partial C_i$, together with the non-retraction theorem.) By \cite[Theorem 8.8]{MunkresDifferential}, if $\delta$ is sufficiently small, then $\tilde h_i^{(\delta)}$ is an embedding with the same image as $\tilde h_i$. Moreover, $\tilde h_i^{(\delta)}$ is isotopic to $\tilde h_i$ via a straight-line homotopy (which is an isotopy by a further application of \cite[Theorem 8.8]{MunkresDifferential}). Thus, $\phi_i^{-1} \circ \tilde h_i^{(\delta)}$ gives a new triangulation of $h_i(D_i) \subset X$ that agrees with the original triangulation away from $h_i(C_i)$, and these can be patched together to give a new triangulation $h'\co T' \to X$ that is isotopic to $h$. By construction, $h'(p_i)$ has a coordinate neighborhood in which $h'(\Sigma)$ is smoothly modeled on the cone of the PL knot $K_i$, as required.
\end{proof}

The second part in the proof is to apply Wall's smoothing theorem for locally flat codimension-2 PL submanifolds of smooth manifolds \cite{WallCodimension2}. Unfortunately, the theorem is only stated in \cite{WallCodimension2} for closed manifolds, but it can easily be adapted to the case of manifolds with boundary. (We presume that this was understood decades ago, but we have not managed to find it written anywhere.) The precise statement is as follows:

\begin{lemma} \label{lemma:smoothing}
Let $V$ be a compact, $n$-dimensional PL manifold, possibly with boundary, and let $V_0$ be a closed collar neighborhood of $\partial V$. Let $M$ be a compact, $(n+2)$-dimensional smooth manifold, possibly with boundary. Let $f \co (V, \partial V) \to (M,\partial M)$ be a piecewise-linear, locally flat embedding. Then any smoothing of  $f|_{V_0}$ can be extended to a smoothing of $f$.
\end{lemma}

\begin{proof}
Wall's proof (in the case where $V$ is closed) relies on obstruction theory for smoothing PL embeddings and immersions that is developed in a two-part paper by Haefliger \cite{HaefligerLissage1, HaefligerLissage2}. The obstruction theory appears in \cite{HaefligerLissage2}, which was never published but has been cited (as ``mimeographed notes'') in numerous other papers. For $V$ closed, Wall notes that by Haefliger's work, the obstructions to smoothing a locally flat embedding of $V$ live in the cohomology groups $H^{i+1}(V;\Gamma_{i-1}^2)$, where $\Gamma_n^q$ denotes the group of concordance classes of smoothings of the standard PL embedding of $S^n$ in $S^{n+q}$. (See \cite{HaefligerLissage1} for the terminology, and specifically Sections 3.5 and 5.1 for the definition and basic properties of these groups.) The main result of Wall's paper is that $\Gamma_n^2=0$ for all $n$; as a consequence, the obstruction automatically vanishes, so any locally flat embedding can be smoothed.

In the relative case, the one-sentence generalization of Wall's argument is that the obstructions to extending a given smoothing of $f|_{V_0}$ live in the relative cohomology groups $H^{i+1}(V, V_0; \Gamma_{i-1}^2)$, which again vanish for the same reason. However, extracting Wall's statement about the cohomology obstruction from Haefliger's paper requires a bit of effort, including in the closed case, so we clarify this here.

For a PL manifold $V$ of dimension $n$, a smooth manifold $M$ of dimension $n+q$, and a locally flat PL embedding $f \co V \to M$, Haefliger \cite[Corollary 11.3]{HaefligerLissage2} states that concordance classes of smoothings of $f$ are in one-to-one correspondence with homotopy classes of sections of a certain bundle with base space $\tilde V$ and fiber $\Gamma^q$, where:
\begin{itemize}
\item The base space $\tilde V$ is a simplicial complex whose $k$-simplices correspond to all PL functions of a $k$-simplex $\Delta^k$ into $V$. (See \cite[p.~8]{HaefligerLissage2} and \cite[p.~83]{HaefligerPoenaru}.)

\item The fiber $\Gamma^q$ is a space with the property that $\pi_i(\Gamma^q) = \Gamma_i^q$ for all $i$, defined in \cite[p.~53]{HaefligerLissage2}.
\end{itemize}
Wall's description of the obstruction in terms of the cohomology of $V$ follows from standard obstruction theory for sections of fiber bundles, along with the observation that there is a natural projection map $\pi \co \tilde V \to V$ inducing an isomorphism on all homology and cohomology groups. (The map $\pi$ is defined on each simplex of $\tilde V$ to be equal to the PL map $\Delta^k \to V$ corresponding to that simplex.) Moreover, in the case where $q=2$, the homotopy groups of the fiber all vanish, so the bundle necessarily admits a section, and thus $f$ admits a smoothing.

Haefliger leaves the relative case of his Corollary 11.3 as an exercise to the reader.\footnote{``Nous laissons au lecteur le soin de formuler l'\'enonc\'e relatif du th\'eor\`eme 11.1 et son corollaire 11.3; il d\'ecoule de la forme relative de 9.3.'' \cite[p.~53]{HaefligerLissage2}} To spell this out, $\tilde V_0$ is naturally a subcomplex of $\tilde V$, and the inclusions $\tilde V_0 \to \tilde V$ and $V_0 \to V$ commute with the projection maps $\pi$, giving an identification $H^*(\tilde V, \tilde V_0) \cong H^*(V,V_0)$ (with any coefficients). The problem of extending a given smoothing of $f|_{V_0}$ is the same as extending a given section defined over $\tilde V_0$. The obstruction theory then proceeds just as above.
\end{proof}

\begin{proof}[Proof of Theorem \ref{thm:PL-smooth}]
By Lemma \ref{lemma:vertexcone}, the surface $\Sigma$ can be assumed to be locally flat except at finitely many singular points $p_1, \dots, p_n$ that are each modeled on cones on PL knots $K_1, \dots, K_n \subset S^3$. This means that there are disjoint closed neighborhoods $V_1, \dots, V_n$, with $v_i \in V_i$, and diffeomorphisms $\phi_i \co (V_i, V_i \cap \Sigma) \to (D^4, \Cone(K_i))$. Let $X' = X - \operatorname{int}(V_1 \cup \dots \cup V_n)$ and $\Sigma' = \Sigma \cap X'$.

Choose a smoothing of $K_i$ in $S^3$ (i.e. an isotopy of $K_i$ to a smooth knot $K'_i$), as well as a smoothing of $\partial \Sigma$ in $\partial X$. Each of these smoothings can be extended to a smoothing on a collar neighborhood of the corresponding boundary component of $\partial \Sigma'$.  Lemma \ref{lemma:smoothing} (with $M' = X$ and $V = \Sigma'$) then shows that the smoothing can be extended over the rest of $\Sigma'$. Moreover, the smoothing of each $K_i$ can be extended radially over $U_i$ to give a smoothing of $(\Sigma \cap U_i) - \{p_i\}$. The resulting surface is smooth away from $p_1, \dots, p_n$, as required.
\end{proof}

\begin{remark} \label{rmk: smoothing}
A possible alternate proof of Theorem \ref{thm:PL-smooth} that avoids the machinery of \cite{HaefligerLissage2} might proceed as follows. First, one would need a stronger version of Lemma \ref{lemma:vertexcone} showing that after a subdivision and perturbation of the triangulation, we can arrange that every vertex of $\Sigma$ has a neighborhood in which $\Sigma$ is smoothly modeled on the cone of a knot, rather than just the vertices where $\Sigma$ is not locally flat. This seems likely to be correct, but we were unable to find a rigorous proof. (The issue is that the technique of approximating by the secant map within a smooth chart, as above, involves subdividing and hence introducing new vertices, so it is unclear how to handle simplices that run from one chosen chart to another.) Assuming this, one could then isotope $\Sigma$ to be smooth in a neighborhood of each of its locally flat vertices (since the knot around each such vertex can be isotoped to a smooth unknot). The final step would then be to smooth $\Sigma$ along its edges by flattening the angles at which pairs of faces of $\Sigma$ meet along edges, which can be done in local coordinates. Note that this would give an alternate proof of Wall's theorem in dimension 4 (i.e., Lemma \ref{lemma:vertexcone} for $n=2$).
\end{remark}

\bibliographystyle{amsalpha}
\bibliography{bib}

\end{document}